\documentclass[a4paper,11pt]{article}
\usepackage[utf8]{inputenc}

\usepackage{amsmath,amssymb,amsthm, amsfonts, epsfig}
\usepackage{tikz}
\usetikzlibrary{decorations.markings}
\usepackage{pifont}

\newcommand{\ee}{\textrm{e}}
\newcommand{\ii}{\textrm{i}}
\newcommand{\dd}{\textrm{d}}
\newcommand{\z}{\zeta}
\newcommand{\om}{\omega}

\usetikzlibrary{patterns}

 \newtheorem{lemma}{Lemma}[section]
 \newtheorem{theorem}[lemma]{Theorem}
 \newtheorem{definition}[lemma]{Definition}

 \newtheorem{remark}[lemma]{Remark}
 
\numberwithin{equation}{section}

\begin{document}

\title{Large degree asymptotics of orthogonal polynomials with respect to an oscillatory weight on a bounded interval}
\author{Alfredo Dea\~{n}o\\[3mm]
Department of Computer Science, KU Leuven\\
Celestijnenlaan 200A, 3001 Heverlee, Belgium\\[3mm]
Departamento de Matem\'aticas, Universidad Carlos III de Madrid\\
Avda. de la universidad, 30. 28911 Legan\'es, Madrid, Spain\\[3mm]
alfredo.deano@cs.kuleuven.be\\
alfredo.deanho@uc3m.es}
\maketitle

\begin{abstract}
We consider polynomials $p_n^{\om}(x)$ that are orthogonal with respect to the oscillatory weight $w(x)=\ee^{\ii\omega x}$ on $[-1,1]$, where $\omega>0$ is a real parameter. A first analysis of $p_n^{\om}(x)$ for large values of $\omega$ was carried out in \cite{ADHW}, in connection with complex Gaussian quadrature rules with uniform good properties in $\omega$. In this contribution we study the existence, asymptotic behavior and asymptotic distribution of the roots of $p_n^{\om}(x)$ in the complex plane as $n\to\infty$. The parameter $\omega$  grows with $n$ linearly. The tools used are logarithmic potential theory and the $S$-property, together with the Riemann--Hilbert formulation and the Deift--Zhou steepest descent method. 
\end{abstract}


\section{Introduction}
In this paper we are concerned with a family of polynomials orthogonal on the interval $[-1,1]$ with respect to an oscillatory weight function. More precisely, we consider the weight function 
\begin{equation}\label{wx}
w(x)=\ee^{\ii\omega x},  
\end{equation}
on $[-1,1]$, where $\omega>0$ is a real parameter, and we define formal orthogonal polynomials, depending on two parameters $\omega$ and $n$, in the following sense:
\begin{equation}\label{OPs}
 \int_{-1}^1 p^{\om}_n(x)x^k w(x)\dd x=0, \qquad k=0, 1,\ldots, n-1.
\end{equation}

The purpose of this paper is to investigate the large $n$ asymptotic behavior of $p^{\om}_n(x)$. We assume that the parameter $\omega$ is coupled linearly with $n$, i.e. $\omega=\omega_n=\lambda n$ for some $\lambda\geq 0$.

Polynomials orthogonal with respect to complex weight functions have been considered in the literature in connection with rational approximation of analytic functions, see for example the work of Aptekarev \cite{Apt}. In that reference, the author considers more general complex weight functions, holomorphic in a neighborhood of the curve where the orthogonality is defined. The weight \eqref{wx} is a particular case of that analysis, although its simplicity allows for more explicit results. It is also important to mention that polynomials with respect to a complex exponential weight have also been considered recently in the work of Suetin \cite{Sue}. In this reference the author studies (using different techniques) the case $\omega=1$ in our notation, modified with a Chebyshev factor $(1-x^2)^{-1/2}$. 

The analysis of this family of orthogonal polynomials was motivated in \cite{ADHW} by the problem of constructing complex quadrature rules of Gaussian type for oscillatory integrals of the type
\begin{equation}\label{oscint}
I[f]= \int_{-1}^1 f(x)\ee^{\ii\omega x}\dd x,
\end{equation}
which is a numerical challenge when the frequency $\omega$ is large. For this and more general Fourier--type integrals, several possibilities have been contemplated in the literature: one option is an application of a Filon--type rule, as exposed for example in \cite{IN}, which is based on interpolation of $f(x)$ and its derivatives (or approximations) at the endpoints. Another possibility is the application of the classical method of steepest descent, see for instance \cite{AH2, HV}, which leads to complex quadrature rules which are optimal for large $\omega$, and whose convergence properties can be improved used interpolation at carefully selected points, see \cite{HO}. In this case, this approach would lead to deformation of the path of integration into the upper complex plane and application of Gauss--Laguerre quadrature to the resulting contour integrals after a suitable parametrization. Another possibility is the so--called exponentially fitted rules, see \cite{LVD}, which are restricted to be real and 
not directly 
connected to orthogonal polynomials.

The main purpose of analyzing the family of orthogonal polynomials $p_n^{\om}(x)$ in \eqref{OPs} is that a complex quadrature rule based on them would have good asymptotic properties both for large $\omega$, resembling the performance of the steepest descent method in that respect, and also for small $\omega$, meaning that it reduces to the classical Gauss--Legendre rule when $\omega\to 0$.

In \cite{ADHW} the authors consider several properties of $p^{\om}_n(x)$, in particular the large $\omega$ asymptotic behavior. The fact that the weight function $w(x)$ is not positive poses a problem of existence of $p_n^{\om}(x)$ from the outset, since the standard Gram--Schmidt procedure to generate the family of orthogonal polynomials may fail for some values of $n$ and/or $\omega$. In fact, one of the conclusions in \cite{ADHW} is that given $n\geq 1$, there exists a countable set of values of $\omega$ for which the polynomial of odd degree $p^{\om}_{2n+1}(x)$ is not defined. This corresponds to zeros of the bilinear form
\begin{equation}
(f,g)=\int_{-1}^1 f(x)g(x)w(x)\dd x
\end{equation}
applied to the polynomial $p^{\om}_{2n}(x)$. For instance, in the case of $p^{\om}_1(x)$, that sequence is given by $\omega=k\pi$, with $k\in\mathbb{Z}$. In contrast, the polynomials of even degree $p^{\om}_{2n}(x)$ are conjectured to exist for all $n$, and under this assumption, their roots tend to the roots of the product of two rotated and scaled Laguerre polynomials as $\omega\to\infty$. 

In order to state asymptotic results for large $n$, and we rewrite the weight as varying with $n$:
\begin{equation}\label{wn}
 w_n(x)=\ee^{-nV(x)}, \qquad V(x)=-\frac{\ii \omega_n x}{n}=-\ii\lambda x. 
\end{equation}


The asymptotic behavior of the roots of $p_n^{\om}(x)$ as $n\to\infty$ will be obtained using logarithmic potential theory and the notion of $S$-curve in a polynomial external field, that goes back to the works of H. Stahl \cite{Stahl}, see also the recent work of Rakhmanov \cite{Rakh}. These ideas have been used recently in connection with non--Hermitian orthogonality in the complex plane, see for instance \cite{KS} for a general formulation of the theory with more general exponential weight functions, see \cite{AMAM,DHK,HKL} for an analysis in the case when $V(x)$ is a cubic polynomial, or the recent contribution \cite{AMFMGT} in connection with Laguerre polynomials with arbitrary complex parameters.

The results on the large $n$ asymptotic behavior of $p_n^{\om}(x)$ will be obtained using the Riemann--Hilbert formulation and the nonlinear steepest descent method. This approach also gives asymptotic information about the coefficients $a^2_n=a^2_{n,\omega}$ and $b_n=b_{n,\omega}$ of the three term recurrence relation
\begin{equation}\label{TTRR}
 xp_n^{\om}(x)=p_{n+1}^{\om}(x)+b_n p_n^{\om}(x)+a^2_n p_{n-1}^{\om}(x),
\end{equation}
which holds provided that $p_n^{\om}(x)$ and $p_{n\pm 1}^{\om}(x)$ are well defined. Here we have omitted the possible dependence of $\omega$ on $n$, of course in this situation the parameter $\omega_n$ would be shifted as well. We also remark that in \cite[Theorem 3.3]{ADHW}, the deformation equations for these coefficients in terms of $\omega$ are given.

Figures \ref{fig_l05}, \ref{fig_l1} and \ref{fig_l15} show the zeros of the polynomials $p_n^{\om}(z)$ in the complex plane, computed with {\sc Maple} using extended precision of 50 digits, and taking $n=20$, $n=40$ and then $\lambda_n=\lambda=1/2$ in Figure \ref{fig_l05}, $\lambda_n=\lambda=1$ in Figure \ref{fig_l1} and $\lambda_n=\lambda=3/2$ in Figure \ref{fig_l15}. These numerical experiments indicate that one should expect a transition in the behavior of the zeros of $p_n^{\om}(x)$, from being supported on a single curve joining $z=-1$ and $z=1$ to distributing along two disjoint arcs in the complex plane, as the coupling parameter $\lambda$ goes through a critical value, say $\lambda_0$, that seems to be located between $1$ and $3/2$. In the next section we make this statement precise.

\begin{figure}
\centerline{
\includegraphics[width=55mm,height=55mm]{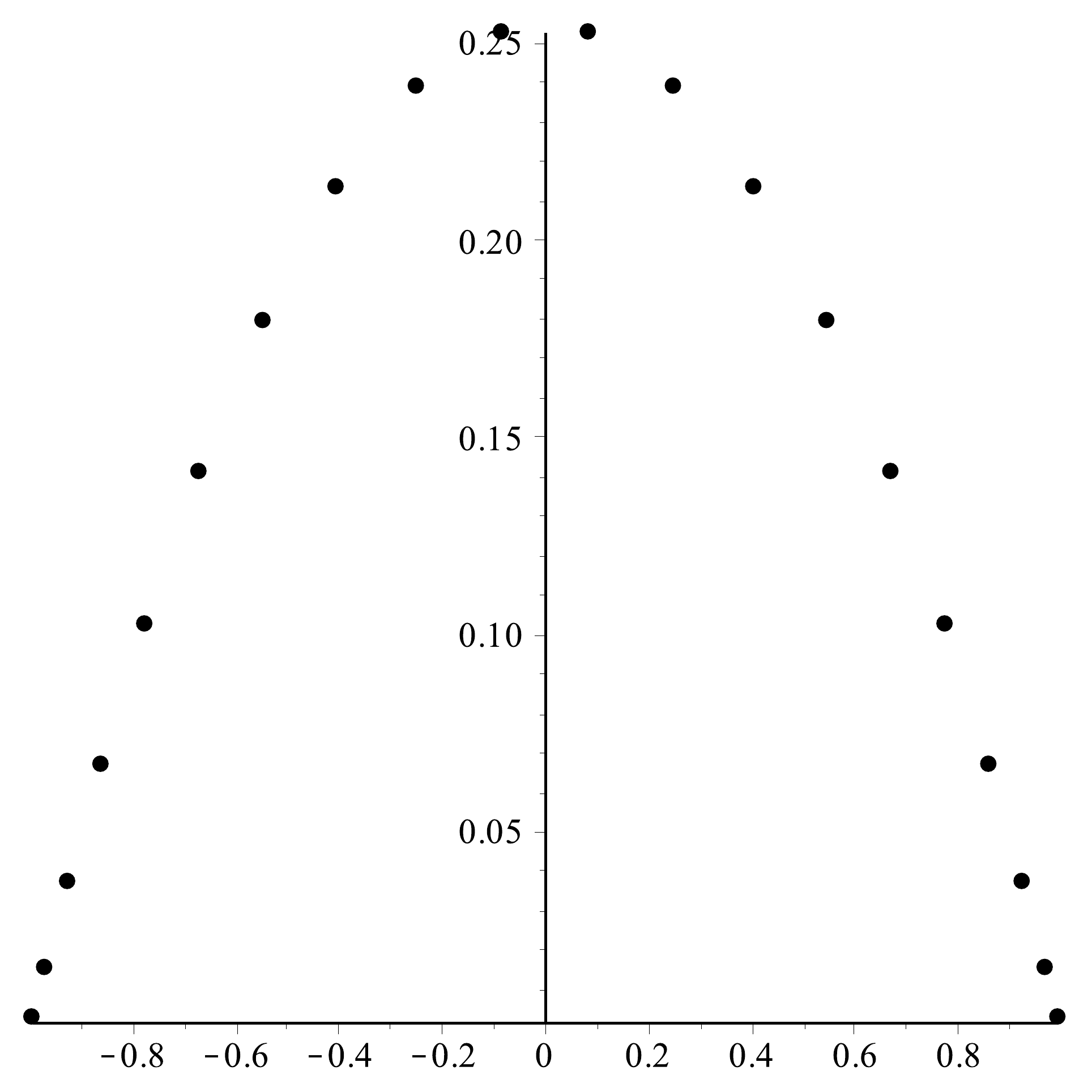}
\hspace{3mm}
\includegraphics[width=55mm,height=55mm]{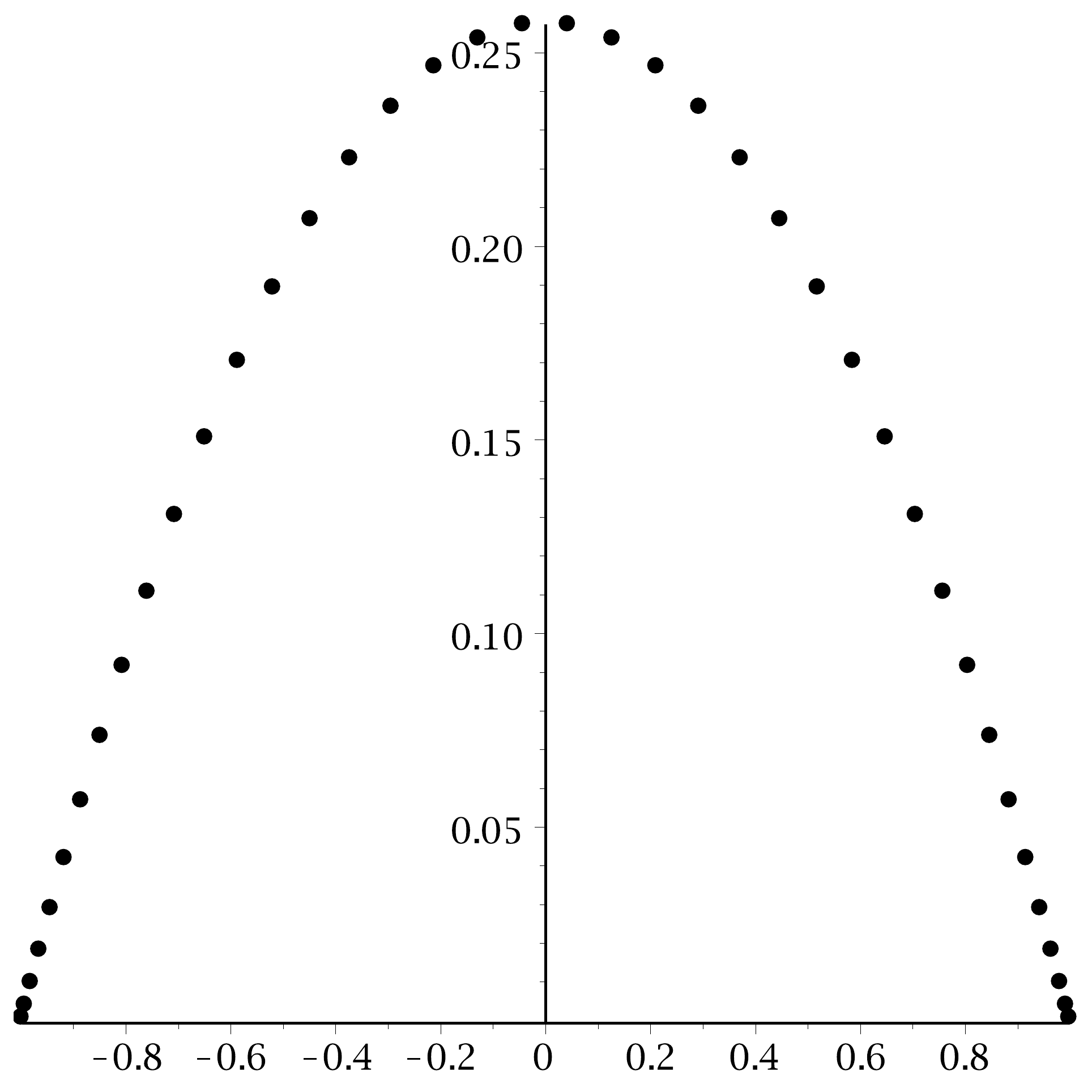}}
\caption{Zeros of $p_{20}^{10}(z)$ (left) and $p_{40}^{20}(z)$  (right). Here $\lambda_n=\lambda=1/2$.}
\label{fig_l05}
\end{figure}

\begin{figure}
\centerline{
\includegraphics[width=55mm,height=55mm]{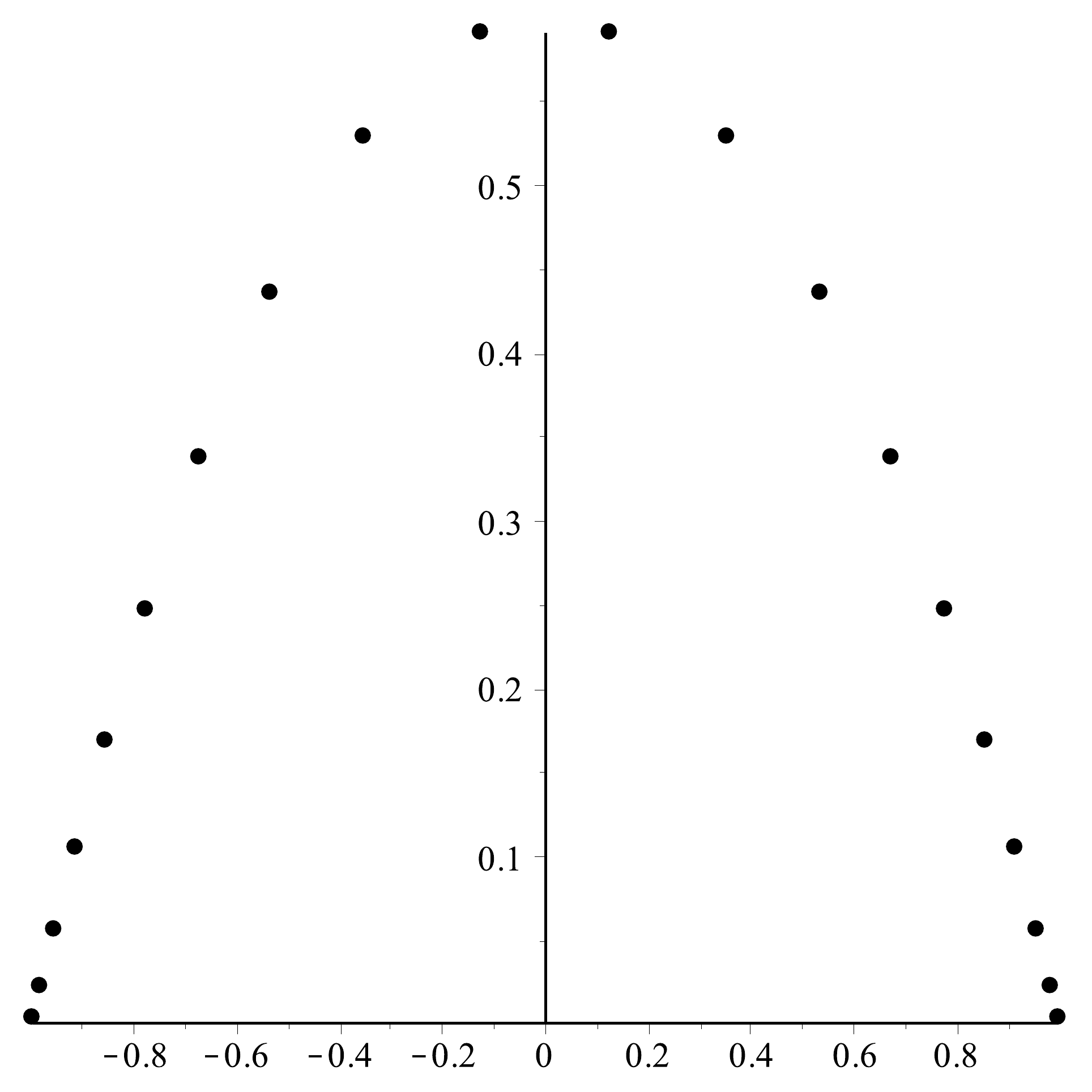}
\hspace{3mm}
\includegraphics[width=55mm,height=55mm]{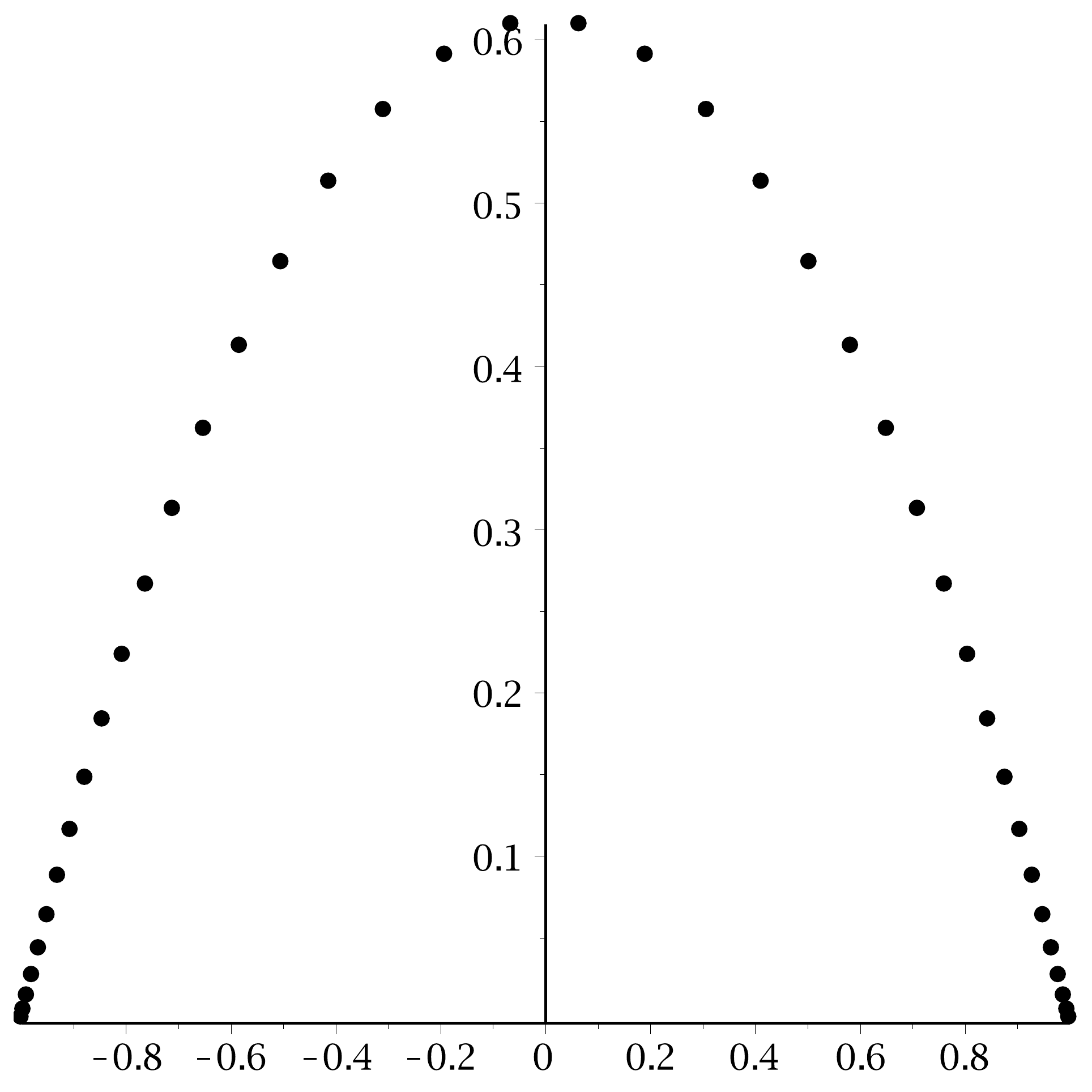}}
\caption{Zeros of $p^{20}_{20}(z)$ (left) and $p^{40}_{40}(z)$ (right). Here $\lambda_n=\lambda=1$.}
\label{fig_l1}
\end{figure}

\begin{figure}
\centerline{
\includegraphics[width=55mm,height=55mm]{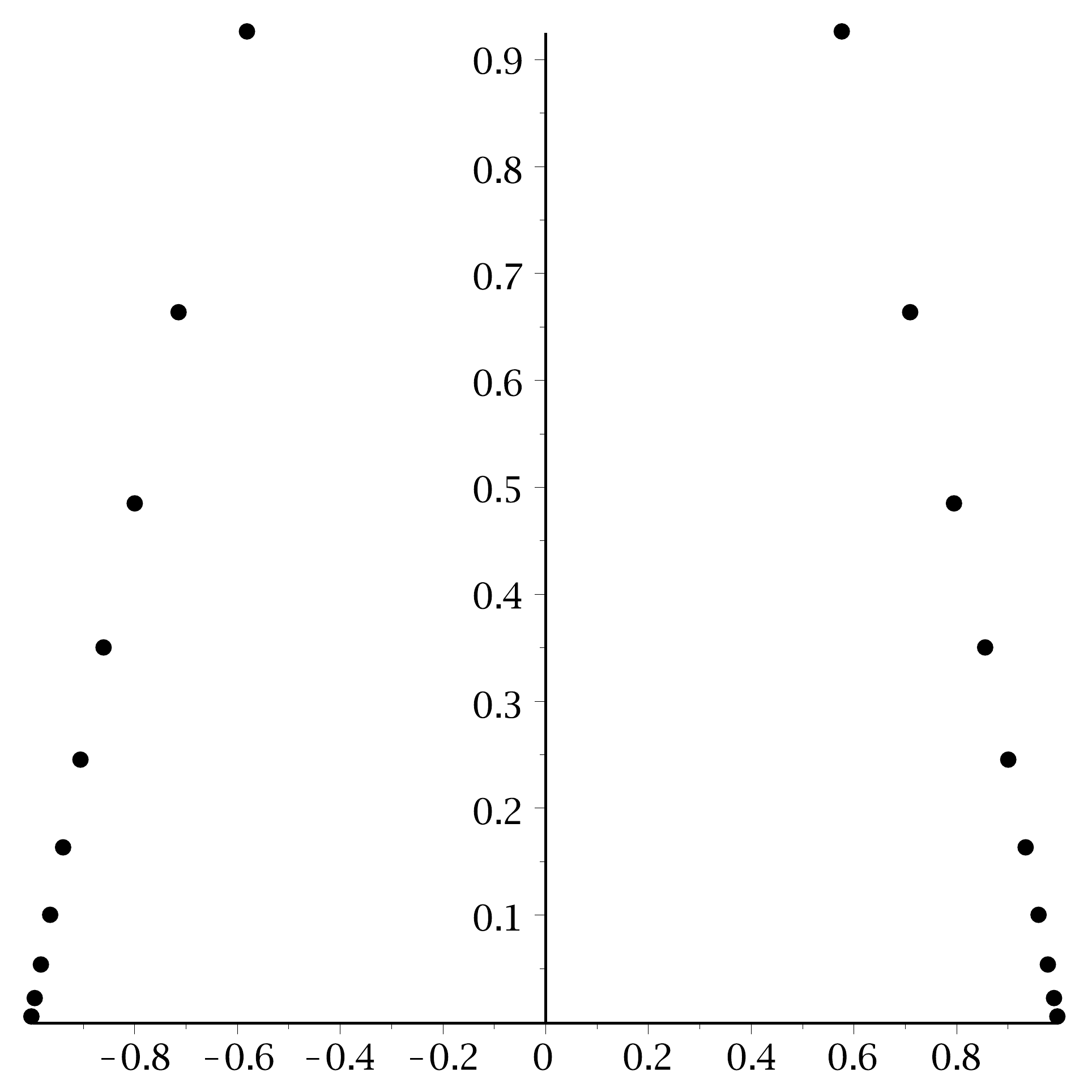}
\hspace{3mm}
\includegraphics[width=55mm,height=55mm]{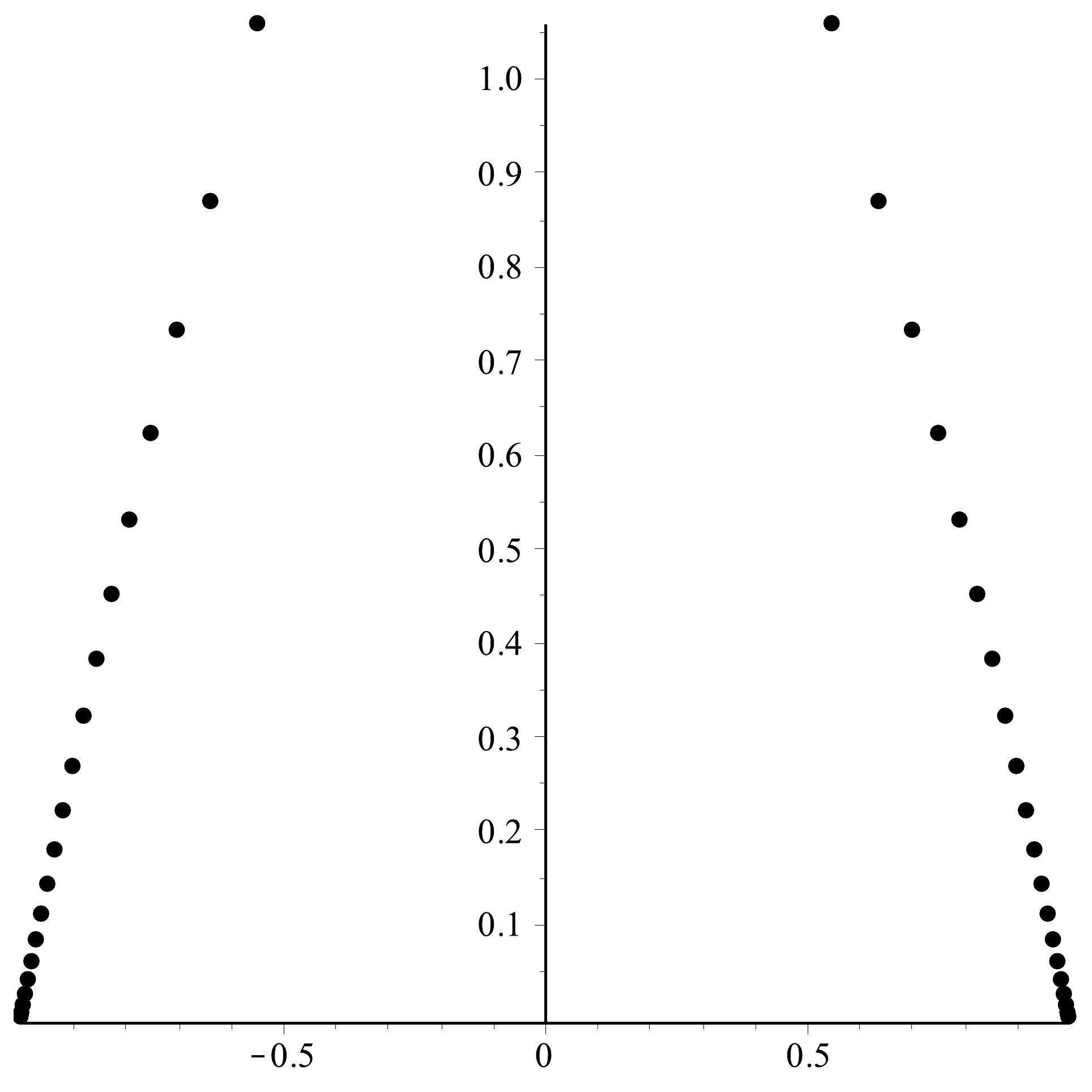}}
\caption{Zeros of $p_{20}^{30}(z)$ (left) and $p_{40}^{60}(z)$ (right). Here $\lambda_n=\lambda=3/2$.}
\label{fig_l15}
\end{figure}

\section{Statement of main results}

Consider $\lambda$ as defined before, and the function
\begin{equation}\label{hl}
 h(\lambda)=2\log\left(\frac{2+\sqrt{\lambda^2+4}}{\lambda}\right)-\sqrt{\lambda^2+4},
\end{equation}
and let $\lambda_0$ be the (unique) positive solution to the equation $h(\lambda)=0$, i.e.
\begin{equation}\label{l0}
 2\log\left(\frac{2+\sqrt{\lambda_0^2+4}}{\lambda_0}\right)-\sqrt{\lambda_0^2+4}=0.
\end{equation}

The fact that $\lambda_0$ is unique is a consequence of $h(\lambda)$ being a map from $[0,\infty)$ onto $\mathbb{R}$ and a decreasing function of $\lambda$. The value of $\lambda_0$ can be computed numerically, and we have $\lambda_0=1.325486839\ldots$

In our results, we will work with the following function:
\begin{equation}\label{varphi}
 \varphi(z)=z+(z^2-1)^2,
\end{equation}
which is analytic in $\mathbb{C}\setminus[-1,1]$ and maps $\mathbb{C}\setminus[-1,1]$ onto the exterior of the unit circle. 

We can prove the following result:
\begin{theorem}\label{Th1}
 Let $V(z)=-\ii\lambda z$ with $0\leq \lambda<\lambda_0$, then 
\begin{enumerate}
\item there exists a smooth curve $\gamma_{\lambda}$ joining $z=1$ and $z=-1$ that is a part of the level set given by
\begin{equation}
\operatorname{Re}\phi(z)=0,
\end{equation}
where
\begin{equation}\label{phiz}
\phi(z)=2\log\varphi(z)+\ii\lambda(z^2-1)^{1/2}
\end{equation}
and the cut of the square root is taken on $\gamma_{\lambda}$,
\item the measure
 \begin{equation}
 \dd\mu_{\lambda}(z)= \psi_{\lambda}(z)\dd z=-\frac{1}{2\pi\ii}\frac{2+\ii\lambda z}{(z^2-1)^{1/2}}\dd z,
 \end{equation}
with a branch cut taken on $\gamma_{\lambda}$, is the equilibrium measure on $\gamma_{\lambda}$ in the external field $\operatorname{Re}\, V(z)$,
 \item the curve $\gamma_{\lambda}$ joining $z=-1$ and $z=1$ has the $S$-property in the external field $\operatorname{Re}\, V(z)$,
 \item if we consider the normalized zero counting measure of $p_n^{\om}(z)$, then
\begin{equation}
 \mu_n=\frac{1}{n}\sum_{p^{\om}_n(\z)=0}\delta(\z)\stackrel{*}{\longrightarrow} \mu_{\lambda},
\end{equation}
as $n\to\infty$, in the sense of weak convergence of measures.
\end{enumerate}
\end{theorem}

In particular, if $\lambda=0$, we get $\textrm{Re}(\log\varphi(z))=0$, and the curve $\gamma_0$ becomes the interval $[-1,1]$. Note that the theorem is consistent with what can be observed in the numerical experiments before.

Regarding the asymptotic behavior of the orthogonal polynomials $p_n^{\om}(z)$, we have the following result:
\begin{theorem}\label{Th2}
Let $0\leq \lambda<\lambda_0$ with $0\leq \lambda<\lambda_0$, then the following holds true:
\begin{enumerate}
 \item For large enough $n$ the orthogonal polynomial $p^{\om}_n(z)$ defined by \eqref{OPs} exists uniquely, and its zeros accumulate on $\gamma_{\lambda}$, as $n\to\infty$. 
 \item For $z\in\mathbb{C}\setminus \gamma_{\lambda}$, the monic orthogonal polynomial $p^{\om}_n(z)$ has the following asymptotic behavior:
\begin{equation}\label{outer_pn}
 p^{\om}_n(z)=\frac{\varphi(z)^{n+1/2}}{2^{n+1/2}(z^2-1)^{1/4}}
\exp\left(-\frac{\ii n\lambda}{2\varphi(z)}\right)
\left(1+\mathcal{O}\left(\frac{1}{n}\right)\right), \quad n\to\infty.
\end{equation}
 \item Fix a neighborhood $U$ of $\gamma_{\lambda}$ in the complex plane, and two discs
\begin{equation}
D(\pm 1,\delta)=\{z\in\mathbb{C}: |z\mp 1|<\delta\},
\end{equation}
with $\delta>0$. For $z\in U\setminus\left(D(1,\delta)\cup D(-1,\delta)\right)$, we have as $n\to\infty$,
\begin{equation}
\begin{aligned}\label{inner_pn}
 p^{\om}_n(z)&=\frac{2^{1/2-n} \ee^{-\frac{\ii n\lambda z}{2}}}{(1-z^2)^{1/4}}\times\\
&\left[\cos\left(\left(n+\frac{1}{2}\right)\arccos z+\frac{n\lambda}{2}(z^2-1)^{1/2}-\frac{\pi}{4}\right)+\mathcal{O}(1/n)\right].
\end{aligned}
\end{equation}
\item For $z\in D(1,\delta)$ we have
\begin{equation}\label{right_pn}
\begin{aligned}
 p^{\om}_n(z)&= 2^{-n}(2 n\pi)^{1/2}f(z)^{1/4}\ee^{-\frac{\ii n\lambda z}{2}} \\
&\times \left[\beta(z)^{-1}J_0\left(-\frac{\ii\, n\phi(z)}{2}\right)-\ii\beta(z) J'_0\left(-\frac{\ii\, n\phi(z)}{2}\right)+\mathcal{O}(1/n)\right], 
\end{aligned}
\end{equation}
as $n\to\infty$, in terms of standard Bessel functions, with $\phi(z)$ given by \eqref{phiz} and 
$f(z)=\phi(z)^2/16$. Here $\beta(z)=\left(\frac{z-1}{z+1}\right)^{1/4}$, with a branch cut taken on $\gamma_{\lambda}$.
\end{enumerate}
\end{theorem}

\begin{remark}
 The asymptotic results are consistent with Theorem 1.4 in \cite{KMcLVAV}, taking $\alpha=\beta=0$ and $h(z)=w_n(z)$, and they can be seen as a complex generalization of the Jacobi--type weight function considered in that reference. Note that the definition of the phase function $\psi(z)$ in \cite[formula (1.34)]{KMcLVAV}, see also \cite[formula (3.9)]{KV} can be adapted to this case. Namely, we have
\begin{equation}
 \psi(z)=\frac{(1-z^2)^{1/2}}{4\pi\ii}\oint_{\gamma} \frac{\log w_n(t)}{(t^2-1)^{1/2}}\frac{\dd t}{t-z}
=\frac{n\lambda}{2}(z^2-1)^{1/2},
\end{equation}
using residue calculation, which fits well with the estimates in Section 3 of the previous theorem. Here $\gamma$ is a smooth curve that encircles $[-1,1]$ once in the positive direction.
\end{remark}

\begin{remark}
 The previous result is also consistent with the strong asymptotics given in Theorem 2 in \cite{Apt}, with notation $Q(z)=-\tfrac{\ii\lambda z}{2}$ and  $\tilde{h}_n(z)\equiv 1$. 
\end{remark}

Regarding the recurrence coefficients in \eqref{TTRR} and the norm of the orthogonal polynomials, we can deduce the following result from the Riemann--Hilbert formulation and the steepest descent analysis:
\begin{theorem}\label{Th3}
Let $0\leq \lambda<\lambda_0$ with $0\leq \lambda<\lambda_0$, then the coefficients $a_n^2$ and $b_n$ in the three term recurrence relation
\begin{equation}
 xp_n^{\om}(x)=p_{n+1}^{\om}(x)+b_n p_n^{\om}(x)+a^2_n p_{n-1}^{\om}(x),
\end{equation}
exist for large enough $n$, and they satisfy
\begin{equation}
 \begin{aligned}
  a_n^2&=\frac{1}{4}+\frac{4-\lambda^2}{4(4+\lambda^2)^2}\frac{1}{n^2}+\mathcal{O}\left(\frac{1}{n^3}\right),\\
  b_n&=-\frac{2\ii\lambda}{(4+\lambda^2)^2}\frac{1}{n^2}+\mathcal{O}\left(\frac{1}{n^3}\right).
 \end{aligned}
\end{equation}
\end{theorem}

\begin{remark}
 Higher order coefficients in the previous asymptotic expansion can in principle be computed by iterating the last step in the steepest descent analysis, see \cite[Section 8]{KMcLVAV} and also Section \ref{sec_Proof3} below.

Following this approach, it is also possible to obtain asymptotic estimates for related quantities, such as leading coefficients of the orthogonal polynomials or Hankel determinants, but we omit these results for brevity.
\end{remark}

\section{Proof of Theorem \ref{Th1}}\label{sec_Proof1}
The proof of Theorem \ref{Th1} will make use of the following tools: first we recall some standard ideas from logarithmic potential theory in the complex plane and the $S$-property, which is a key tool in the analysis of non--Hermitian orthogonality in the complex plane. In order to construct such a curve with the $S$-property, which will attract the zeros of $p_n^{\om}(z)$ as $n\to\infty$, we need to study the local and global trajectories of a certain quadratic differential $-Q_{\lambda}(z)\dd z^2$, that turns out to be explicit in this case. 

\subsection{Potential theory and the $S$-property}

From the work of Gonchar and Rakhmanov, \cite{GR,Rakh}, it is known that a key element in the analysis of the large $n$ behavior of zeros of orthogonal polynomials is the equilibrium measure in an external field, which in this case is  given by $\textrm{Re}\, V(z)$, with $V(z)=-\ii\lambda z$. 

We take the set $\mathcal{T}$ of smooth curves joining $z=-1 $ with $z=1$, with orientation from $-1$ to $1$. For any $\gamma\in\mathcal{T}$ we consider the space of probability measures on $\gamma$, denoted here $\mathcal{M}_1(\gamma)$, and we pose the following equilibrium problem: 
\begin{equation}
 \textrm{inf}\{I(\mu): \mu\in\mathcal{M}_1(\gamma)\},
\end{equation}
where 
\begin{equation}
 I(\mu)=\int_{\gamma} U^{\mu}(z)\dd\mu(z)+\textrm{Re}\int_{\gamma} V(z)\dd\mu(z)
\end{equation}
is the energy functional and
\begin{equation}
 U^{\mu}(z)=\int_{\gamma}\log\frac{1}{|z-y|}\dd\mu(y)
\end{equation}
is the logarithmic potential of the measure $\mu$. This problem has a unique solution for each $\gamma\in\mathcal{T}$, following the standard theory, see for instance \cite{ST}. Furthermore, the equilibrium measure $\mu$ is characterized by the so--called variational conditions. Define
\begin{equation}
 g(z)=\int_{\gamma} \log (z-s)\dd\mu(s), 
\end{equation}
which is analytic in $\mathbb{C}\setminus\gamma$. Note that $\textrm{Re}\, g(z)=-U^{\mu}(z)$. Then there exists a constant $\ell$ such that 
\begin{equation}\label{variational}
 \begin{aligned}
   \textrm{Re}\left(-g_+(z)-g_-(z)+V(z)\right) &=\ell, \qquad z\in\textrm{supp}\,\mu,\\
   \textrm{Re}\left(-g_+(z)-g_-(z)+V(z)\right)&\geq\ell, \qquad z\in\gamma.
 \end{aligned}
\end{equation}
 
When working with non--Hermitian orthogonality in the complex plane, we have an additional freedom to choose $\gamma$ without changing the orthogonality condition \eqref{OPs}. In order to find the precise curve that describes the asymptotic behavior of the zeros of $p^{\om}_n(x)$ as $n\to\infty$, we need an extra condition, which is called the $S$-property in the literature. We refer the reader to \cite{Stahl,Rakh,KS} for more details.

\begin{definition}
We say that a contour $\gamma$ has the $S$-property in the external field $\operatorname{Re}\, V(x)$ if the following two conditions are satisfied:
\begin{enumerate}
 \item there exists a set $E$ of zero capacity such that for any $\z\in \gamma\setminus E$, there exists a neighborhood $D=D(\z)$ such that the set $D\,\cap\,\operatorname{supp}\,\mu$ is an analytic arc,
 \item for any point on this analytic arc, the logarithmic potential of $\mu$ satisfies
\begin{equation}\label{Sprop}
 \frac{\partial }{\partial n_+}\left(U^{\mu}+\frac{1}{2}\operatorname{Re}\, V\right)=
 \frac{\partial }{\partial n_-}\left(U^{\mu}+\frac{1}{2}\operatorname{Re}\, V\right),
\end{equation}
where $\frac{\partial}{\partial n_{\pm}}$ indicates normal derivatives in opposite directions, and $\mu$ is the equilibrium measure for $\gamma$ in the external field $\operatorname{Re}\, V$.
\end{enumerate}
 \end{definition}

If one is able to find such a contour with the $S$-property in the external field $\textrm{Re}\,V(z)$, then it follows from the work of Gonchar and Rakhmanov, cf. \cite[Theorem 1.3]{Rakh}, that the zero counting measure on the zeros of $p^{\om}_n(x)$ satisfies
\begin{equation}
 \mu_n=\frac{1}{n}\sum_{p^{\om}_n(\z)=0}\delta(\z)\stackrel{*}{\longrightarrow} \mu,
\end{equation}
where $\mu$ is the equilibrium measure of $\gamma$ in the external field $\operatorname{Re}\, V$. Here $\stackrel{*}{\longrightarrow}$ indicates convergence in the weak topology, in the sense that
\begin{equation}
 \int_{\gamma} f(x)\dd\mu_n(x)\rightarrow \int_{\gamma} f(x)\dd\mu(x), \qquad n\to\infty,
\end{equation}
for any continuous function $f$. 

In the sequel we will denote the curve with the $S$-property as $\gamma_{\lambda}$, depending on the parameter $\lambda$, and $\mu_{\lambda}$ will be the equilibrium measure on $\gamma_{\lambda}$.

An alternative characterization of the curve $\gamma_{\lambda}$ with the $S$-property is a consequence of the Cauchy--Riemann equations, and has a form similar to the variational equations \eqref{variational}, but involving the imaginary part of the external field:
\begin{equation}\label{Sprop_Im}
   \textrm{Im}\left(-g_+(z)-g_-(z)+\frac{1}{2}V(z)\right) =\tilde{\ell}, \qquad z\in\textrm{supp}\,\mu_{\lambda},
\end{equation}
where the constant $\tilde{\ell}$ might be different in each component of $\textrm{supp}\,\mu$, if this consists of several pieces. Hence, combining \eqref{Sprop_Im} and \eqref{variational}, if $\gamma_{\lambda}$ has the $S$-property in the external field $\textrm{Re}\, V(z)$, we have
\begin{equation}\label{variational2}
   -g_+(z)-g_-(z)+V(z) =\ell+\ii\tilde{\ell}, \qquad z\in\textrm{supp}\,\mu_{\lambda}.
\end{equation}

Furthermore, it is known, see for example \cite{Rakh}, that the support of the equilibrium measure $\mu_{\lambda}$ is a union of analytic arcs that are trajectories of the quadratic differential $-Q_{\lambda}(z) \dd z^2$, which are given by the condition
\begin{equation}
-Q_{\lambda}(z) \dd z^2>0. 
\end{equation}

This function $Q_{\lambda}(z)$ is defined as
\begin{equation}\label{Rz_general}
 Q_{\lambda}(z)=\left(\int\frac{\dd \mu_{\lambda}(x)}{x-z}+\frac{V'(z)}{2}\right)^2.
\end{equation}

This equation is a key element in order to obtain properties of the equilibrium measure and its Cauchy transform, see for example \cite{MFR, KS}, but in order to obtain some explicit formula for the quadratic differential,  one normally has to make some assumptions on $\mu_{\lambda}$, to be proved later. Assume for the moment that $\mu_{\lambda}$ is supported on a single arc $\gamma_{\lambda}$ that connects $z=-1$ and $z=1$ in the complex plane, then the function
\begin{equation}
 w_{\lambda}(z)=\int_{\gamma_{\lambda}}\frac{\dd \mu_{\lambda}(x)}{z-x}
\end{equation}
is analytic in $\mathbb{C}\setminus\gamma_{\lambda}$, and it satisfies 
\begin{equation}\label{cond_w}
 \begin{aligned}
  w_{\lambda}(z)&=\frac{1}{z}+\mathcal{O}\left(\frac{1}{z^2}\right), \qquad z\to\infty,\\
  w_{\lambda +}(z)+w_{\lambda -}(z)&=V'(z)=-\ii\lambda, \qquad z\in\gamma_{\lambda},
 \end{aligned}
\end{equation}
where $w_{\pm}(z)$ denote the boundary values from the left (right) of the curve $\gamma_{\lambda}$. The second equation is a direct consequence of the formula \eqref{variational2}, since $w_{\lambda}(z)=g'(z)$.

Consequently, we look for $w_{\lambda}(z)$ in the form
\begin{equation}\label{ansatzw}
 w_{\lambda}(z)=-\frac{\ii\lambda}{2}+\frac{p_1(z)}{(z^2-1)^{1/2}},
\end{equation}
where $p_1(z)=a_0+a_1z$ is a polynomial of degree 1, because of the first condition in \eqref{cond_w}, and the square root is taken with a cut on $\gamma_{\lambda}$. From the first equation in \eqref{cond_w}, we obtain $a_0=1$ and $a_1=\ii\lambda/2$, so
\begin{equation}
 w_{\lambda}(z)=-\frac{\ii\lambda}{2}+\frac{2+\ii\lambda z}{2(z^2-1)^{1/2}}.
\end{equation}

Then, for $z\in\gamma_{\lambda}$, we have
\begin{equation}\label{wpm}
 \begin{aligned}
  w_{\lambda\pm}(z)&=-\frac{\ii\lambda}{2}\mp\ii\, \frac{2+\ii\lambda z}{2\sqrt{1-z^2}},\\
 \end{aligned}
\end{equation}
where the branch of the square root that is real and positive on $\gamma_{\lambda}$. The density of the equilibrium measure can be recovered as
\begin{equation}\label{psiz}
 \dd \mu_{\lambda}(z)=\psi_{\lambda}(z)\dd z=-\frac{1}{2\pi\ii}\left(w_{\lambda +}(z)-w_{\lambda -}(z)\right)\dd z=\frac{1}{2\pi}\frac{2+\ii\lambda z}{\sqrt{1-z^2}}\dd z,
\end{equation}
using the Sokhotski--Plemelj formula, see for instance \cite[\S 1.4.2]{Gak}. The density can be extended to the complex plane with a cut on $\gamma_{\lambda}$:
\begin{equation}\label{psi}
  \dd \mu_{\lambda}(z)=\psi_{\lambda}(z)\dd z=-\frac{1}{2\pi\ii}\frac{2+\ii\lambda z}{(z^2-1)^{1/2}}\dd z.
\end{equation}

Finally, the quadratic differential $Q_{\lambda}(z)\dd z^2$ has the form
\begin{equation}\label{Rz}
 Q_{\lambda}(z)\dd z^2=\frac{(2+\ii\lambda z)^2}{4(z^2-1)}\dd z^2.
\end{equation}

In general, this argument will break down if the support of $\mu_{\lambda}$ has several components, in particular \eqref{ansatzw} will cease to be valid. For this reason, in the next section we will analyze this quadratic differential in more detail, with the goal of showing that there is indeed one trajectory of $-Q_{\lambda}(z)\dd z^2$ joining $z=-1$ and $z=1$, for small values of $\lambda$. This will prove the previous heuristic argument.

\subsection{Trajectories of the quadratic differential $-Q_{\lambda}(z)\dd z^2$}
\subsubsection{Local trajectories near critical points}
From \eqref{Rz}, it is clear that the finite critical points of the quadratic differential $-Q_{\lambda}(z)\dd z^2$ are two poles located at $z=\pm 1$, a double zero at $z=z^*=2\ii/\lambda$. At infinity, the quadratic differential has a pole of order $4$, since using the local parameter $\xi=1/w$, we get
\begin{equation}
\begin{aligned}
-\frac{1}{w^4} Q_{\lambda}\left(\frac{1}{w}\right) \dd w^2
&=\left(\frac{\lambda^2}{4w^4}-\frac{\ii\lambda}{w^3}+\frac{\lambda^2-4}{4w^2}-\frac{\ii\lambda}{w}+\mathcal{O}(1)\right) \dd w^2,\qquad w\to 0.
\end{aligned}
\end{equation}

Naturally, if $\lambda=0$, the double zero disappears and the pole at infinity becomes of order $2$, with negative residual.

In a neighborhood of a critical point the condition $-Q_{\lambda}(z) \dd z^2>0$ is known to be equivalent, see \cite[Section 4]{MFR}, to $\textrm{Im}\,\xi_{\lambda}(z)=\textrm{const}$, in terms of the local parameter $\xi_{\lambda}$, given by
\begin{equation}
 \xi_{\lambda}=\xi_{\lambda}(z)=\int^z \sqrt{-Q_{\lambda}(s)}\dd s.
\end{equation}

In our case, this local parameter can be computed explicitly:
\begin{equation}\label{xiz}
 \xi_{\lambda}(z)= \ii\log (z+(z^2-1)^{1/2})-\frac{\lambda(z^2-1)^{1/2}}{2},
\end{equation}
with a branch cut on $(-\infty,-1]\cup\gamma_{\lambda}$ and the principal value of the root and the logarithm.

From \eqref{xiz}, we have the following local behavior near $z=1$:
\begin{equation}
 \xi_{\lambda}(z)-\xi_{\lambda}(1)= \frac{\sqrt{2}}{2\pi}(-\lambda+2\ii)(z-1)^{1/2}+\mathcal{O}((z-1)^{3/2}).
\end{equation}

If we write $z=1+r\ee^{\ii\theta}$, the argument of the leading term is 
\begin{equation}
 \arg (-\lambda+2\ii)(z-1)^{1/2}=\frac{\theta}{2}-\arctan \frac{2}{\lambda}.	
\end{equation}

If this is to be real and positive, we get
\begin{equation}
  \theta \approx 2\arctan \frac{2}{\lambda},
\end{equation}
taking principal values of the argument. If $\lambda\to 0^+$ we get $\theta\to \pi^-$, so the trajectory follows the real axis from $z=1$. The angle decreases as we increase the value of $\lambda$, so the trajectory enters the upper half plane. By symmetry, we have a similar behavior near $z=-1$.

If $\lambda>0$, the quadratic differential has a double zero, located at
\begin{equation}
z^*=\frac{2\ii}{\lambda}.
\end{equation}

Locally, four trajectories emanate from $z=z^*$. Note that
\begin{equation}
 \xi_{\lambda}(z)-\xi_{\lambda}(z^*)=\frac{\ii\lambda^2}{4\sqrt{\lambda^2+4}}(z-z^*)^2+\mathcal{O}((z-z^*)^3).
\end{equation}

So, if we write $z=z^*+r\ee^{\ii\theta}$ and we take arguments in the previous equality, we get
\begin{equation}
 \frac{\pi}{2}+2\theta\equiv 0\,\, (\textrm{mod}\, \pi), 
\end{equation}
which gives the possible angles 
\begin{equation}
 \theta=\pm\frac{\pi}{4}, \qquad
 \theta=\pm\frac{3\pi}{4}.
\end{equation}

Regarding infinity, the pole of order $4$ attracts the trajectories in two directions with angle $\pi$.

\subsubsection{Global trajectories}
The main tool that we will use to analyze the global behavior of trajectories of $-Q_{\lambda}(z)\dd z^2$ is the Teichm\"{u}ller lemma, \cite[Theorem 14.1]{Strebel}. First we define the following:
\begin{definition}
 A $Q_{\lambda}$-polygon is a curve $\Gamma$ composed of open straight arcs, along which $\arg Q_{\lambda}(z)\dd z^2=\theta=\textrm{const}$, with $0\leq\theta\leq 2\pi$ and their endpoints, which can be critical points.
\end{definition}

In particular, we can use horizontal ($\theta=0$) and vertical ($\theta=\pi$) trajectories of the quadratic differential as boundaries of a $Q_{\lambda}$-polygon. 

Let us denote the zeros and poles of $Q_{\lambda}(z)$ as $z_i$, and assign a number $n_i$ to each of them, so that $n_i$ is positive and equal to the order of the zero if $z_i$ is a zero, and $n_i$ is negative and equal to the order of the pole if $z_i$ is a pole. Then, we have the following result, see \cite[Theorem 14.1]{Strebel} for the proof:
\begin{theorem} 
Let $D$ be the interior of a simple closed $Q_{\lambda}$-polygon $\Gamma$, with sides $\Gamma_j$ and interior angles $\theta_j$ at its vertices, $0\leq \theta\leq 2\pi$, and suppose that $Q_{\lambda}(z)$ is meromorphic in $\overline{D}$ (the only points which can be critical on the boundary of the polygon are the vertices). Then
\begin{equation}
 \sum_{j}\left(1-\theta_j\frac{n_j+2}{2\pi}\right)=2+\sum_{i} n_i, 
\end{equation}
where the index $i$ runs over all critical points inside the sector.
\end{theorem}

In the present situation, consider first the four trajectories emanating from the double zero $z^*=2\ii/\lambda$, with starting angles equal to $\pi/2$, see Figure \ref{fig_traject1}. 

\begin{figure}[h]
\centerline{\includegraphics[width=75mm,height=60mm]{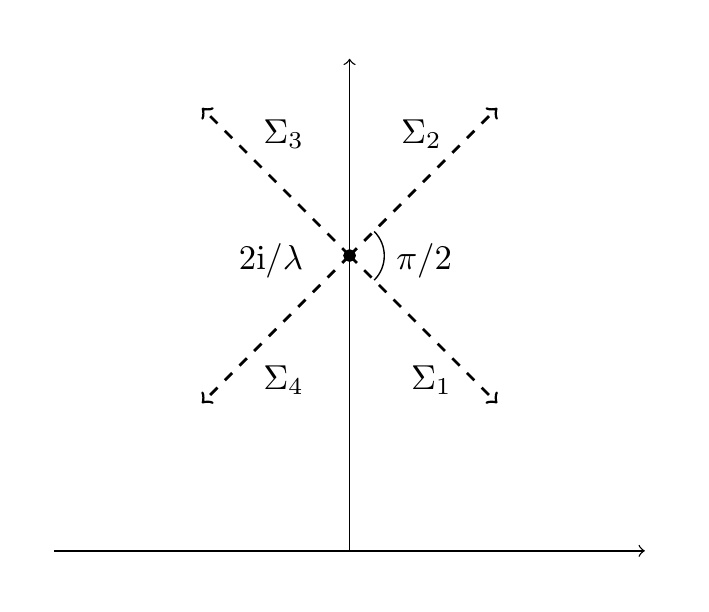}}
\caption{Local trajectories emanating from $z=2\ii/\lambda$.}
\label{fig_traject1}
\end{figure}

Assume first that the four trajectories from $z=z^*$ diverge to infinity, and consider one of the sectors delimited by two adjacent trajectories, say $\Omega$. The internal angle at the double zero $z=z^*$ is $\pi/2$, so we have
\begin{equation}
 1+\frac{\theta}{\pi}=2+\sum_{i} n_i,
\end{equation}
where $\theta$ is the angle at infinity. If both poles $z=\pm 1$ are in $\Omega$, then the right hand side of the equation is $0$, which does not give a valid value for $\theta$. If there are no poles in $\Omega$ then the right hand side is $2$, which is possible, but clearly cannot happen in the four sectors, since the poles must be in one of them.

Therefore, we can have two situations: 
\begin{enumerate}
 \item The trajectories define four infinite sectors, two of which are free of poles, one contains $z=-1$ and another one contains $z=1$, symmetric with respect to the imaginary axis.
 \item The trajectories define two infinite sectors and a finite one. If we have a closed trajectory from $z=z^*$, then we get
\begin{equation}
 0=2+\sum_{i} n_i,
\end{equation}
 so both poles $z=\pm 1$ must be inside the loop. So in this case, the finite trajectory encircles the two poles and the other two regions are infinite.
\end{enumerate}

An intermediate situation can take place: two trajectories go to infinity and the other two connect $z=z^*$ with $z=1$ and with $z=-1$. This should take place when $\textrm{Im}\, \xi(1)=\textrm{Im}\, \xi(z^*)$, and since $\textrm{Im}\, \xi(1)=0$, we get a value of $\lambda$, say $\lambda_0$, that must satisfy the equation
\begin{equation}\label{eq_l0}
 2\log\left(\frac{2+\sqrt{\lambda_0^2+4}}{\lambda_0}\right)-\sqrt{\lambda_0^2+4}=0.
\end{equation}

Thus, this formula defines the value $\lambda_0$ that separates the two different cases in the behavior of the zeros of $p_n^{\om}(x)$. Namely, as a consequence of the previous analysis, we have the following result:
\begin{lemma}
 Let $\lambda_0$ be defined as the unique positive solution of \eqref{eq_l0}. Then
\begin{itemize}
 \item If $0\leq\lambda<\lambda_0$, then there exists a trajectory of the quadratic differential $-Q_{\lambda}(z)\dd z^2$ that joins the points $z=1$ and $z=-1$. 
 \item If $\lambda>\lambda_0$, then the trajectories of the quadratic differential $-Q_{\lambda}(z)\dd z^2$ emanating from $z=1$ and $z=-1$ diverge to infinity in opposite directions.
\end{itemize}
\end{lemma}

\begin{proof}
In the case $\lambda=0$ it is clear that $\textrm{Im}\,\xi_0(z)=\textrm{Im}\,\xi_0(1)=0$ is equivalent to $\textrm{Re}( \log\varphi(z))=0$. Since the function $\varphi(z)$ maps the interval $[-1,1]$ onto the unit circle, so $\log\varphi(z)$ is purely imaginary for $z\in[-1,1]$. This comes as no surprise, because the case $\lambda=0$ corresponds to the classical Legendre polynomials.

If $\lambda>0$, let us consider the function $\operatorname{Im}(\xi(z))$ when $z=x\geq 0$ is on the positive real axis. From \eqref{xiz}, we get that if $x\geq 1$, then
\begin{equation}
 \operatorname{Im}(\xi(x))=\log(x+\sqrt{x^2-1}),
\end{equation}
which takes the value $0$ at $x=1$ and is increasing with $x$. If $0\leq x\leq 1$, then $\varphi(x)=x+\sqrt{x^2-1}$ maps the interval $[0,1]$ onto the arc of the unit circle $\operatorname{e}^{\operatorname{i}\theta}$, with $0\leq \theta\leq \pi/2$, so $\operatorname{i}\log(x+\sqrt{x^2-1})\in\mathbb{R}$, and therefore
\begin{equation}
 \operatorname{Im}(\xi(x))=-\frac{\lambda\sqrt{1-x^2}}{2}, \qquad 0\leq x \leq 1,
\end{equation}
which takes the value $-\lambda/2<0$ at $x=0$ and increases to $0$ when $x=1$. As a consequence, this imaginary part is a continuous and increasing function of $x$ on the positive real axis.

Now consider $\operatorname{Im}\,(\xi(z^*))$, which is equal to
\begin{equation}
 \operatorname{Im}\, \xi\left(\frac{2\operatorname{i}}{\lambda}\right)
=\log\left(\frac{\sqrt{\lambda^2+4}+2}{\lambda}\right)-\frac{\sqrt{\lambda^2+4}}{2}.
\end{equation}

This function maps $[0,\infty)$ onto $\mathbb{R}$, with value $0$ at $\lambda_0$, and it is a decreasing function of $\lambda$. 

Therefore, because of continuity and monotonicity, if we fix a value of $\lambda$, there exists a (unique) real and positive value $x$ such that $\operatorname{Im}(\xi(x))=\operatorname{Im}(\xi(z^*))$. If $\lambda>\lambda_0$, then $0<x<1$, and if $0<\lambda<\lambda_0$, then $x>1$. 

If $0<\lambda<\lambda_0$ then the trajectory $\Sigma_1$ crosses the real axis at a point $x>1$, and we have a closed trajectory starting and ending at $z^*$ and encircling the two poles $z=\pm 1$. In this case, there is a trajectory emanating from $z=1$, and since trajectories cannot intersect, it must end at the other pole $z=-1$. This proves that if $0<\lambda<\lambda_0$, there is a curve joining the two critical points $z=\pm 1$.

If $\lambda>\lambda_0$, the trajectory $\Sigma_1$ from $z^*$ crosses the real axis at a point $x<1$, and we must be in the first case before. We apply the Teichm\"{u}ller lemma to the sectors bounded by $\Sigma_2$ and $\Sigma_3$ and by $\Sigma_4$ and $\Sigma_1$. In both cases, we get that the angle at infinity is equal to $\pi$, so the two trajectories diverge to infinity in opposite directions. In the two remaining sectors, the angle at infinity is $0$, so the trajectory emanating from $z=\pm 1$ must diverge to infinity within these sectors. 
 
\end{proof}

This proves the first part of Theorem \ref{Th1}, since a trajectory of the quadratic differential satisfies $-Q_{\lambda}(z)\dd z^2>0$, which is equivalent to $\textrm{Im}\,\xi(z)=\textrm{const}$, the constant being equal to $\textrm{Im}\,\xi(1)=0$ in this case. Bearing in mind \eqref{xiz}, we get the desired result.

Next, because of the explicit construction of $\dd\mu_{\lambda}$ before, see \eqref{psi} and the discussion leading to this formula, it follows that \eqref{variational2} are satisfied, so $\gamma_{\lambda}$ has the $S$-property in the external field $\textrm{Re}\, V(z)$. Finally, the weak convergence of the zero counting measure for $p^{\om}_n(z)$ is a direct consequence of the work of Gonchar and Rakhmanov, and this completes the proof of Theorem \ref{Th1}.

\begin{figure}
\centerline{
\includegraphics[width=43mm,height=63mm]{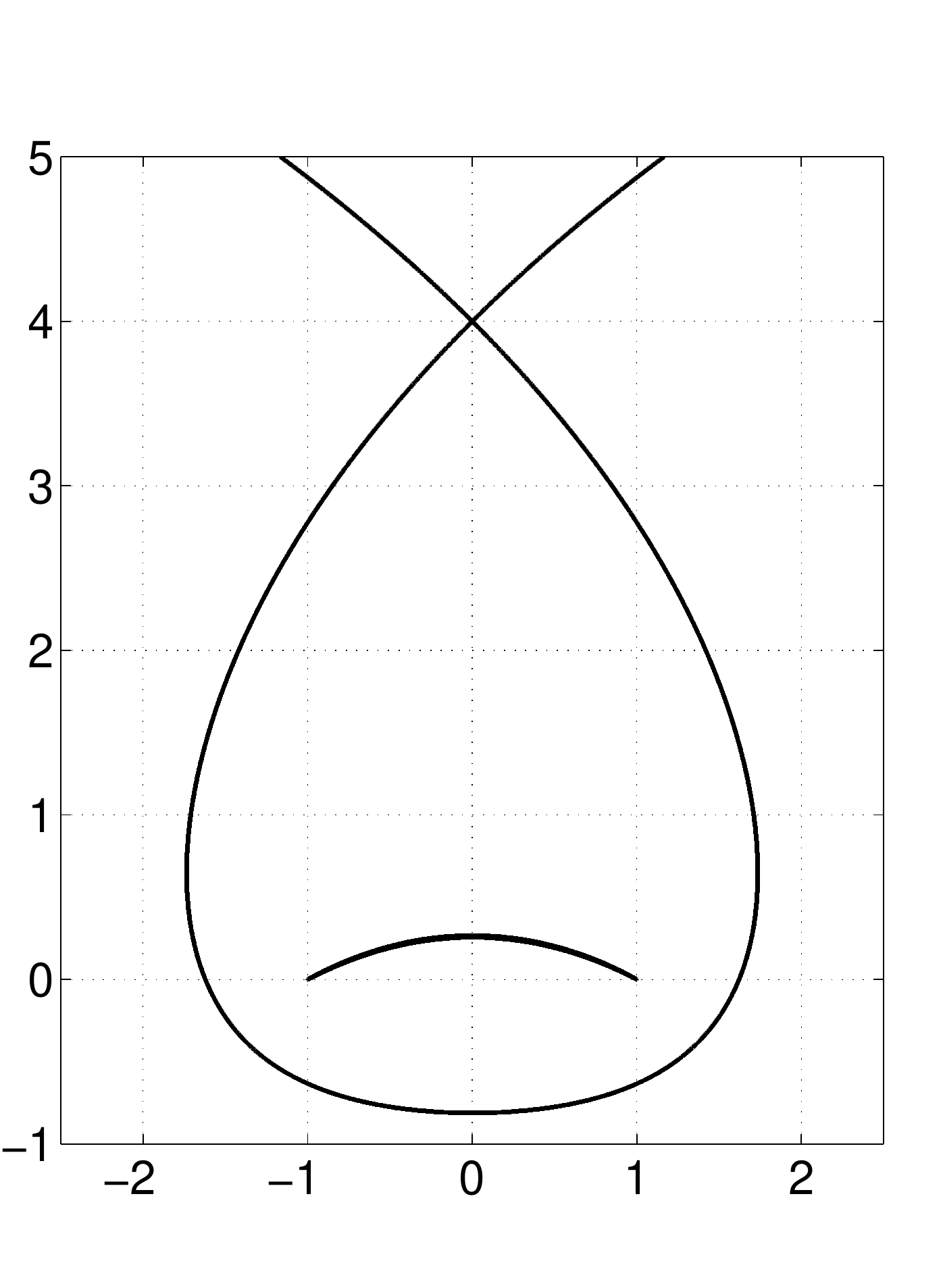}
\includegraphics[width=43mm,height=60mm]{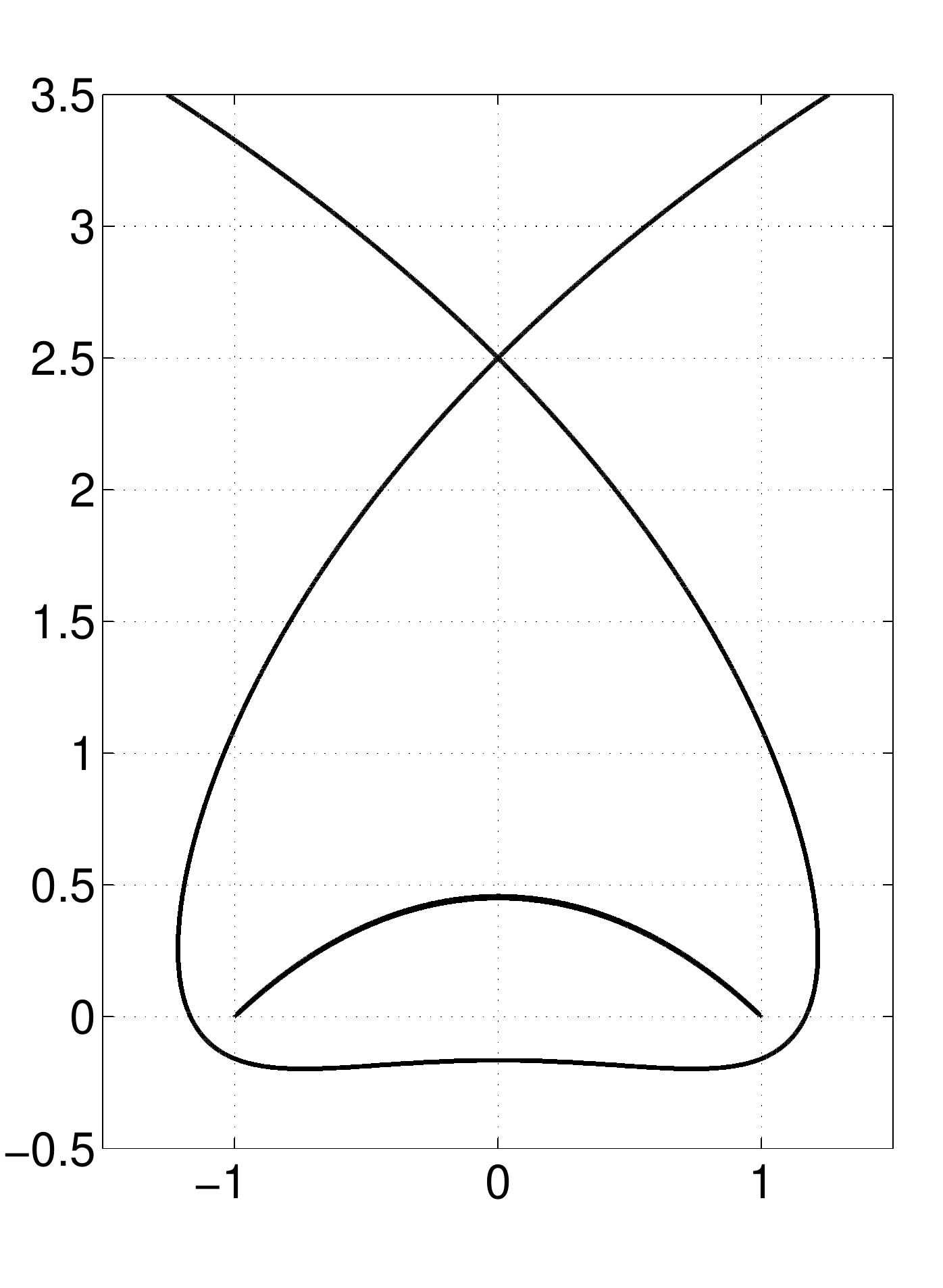}
\includegraphics[width=43mm,height=60mm]{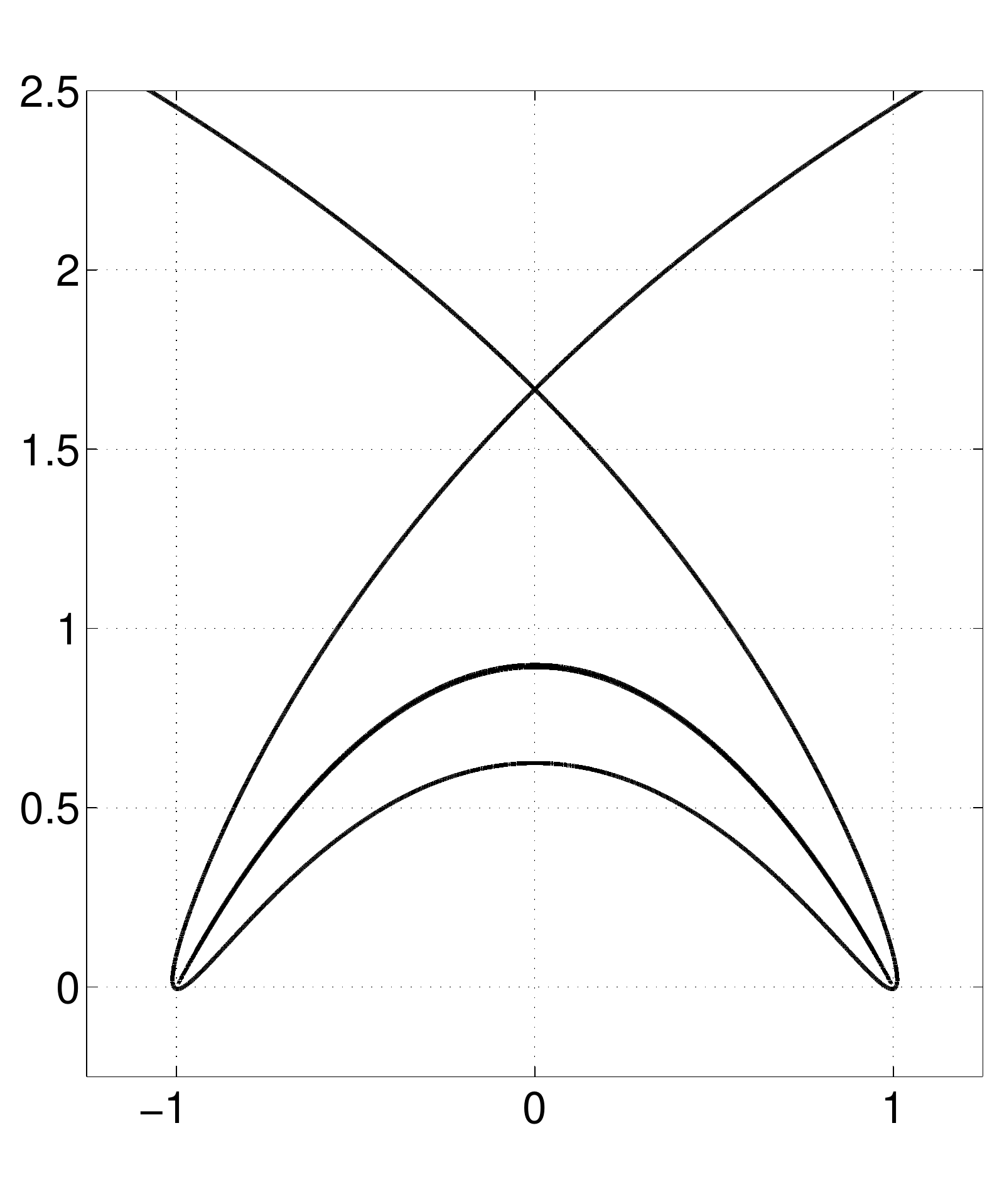}
}
\caption{Trajectories corresponding to the critical points $z=\pm 1$ (in black) and to the point $z=z^*$ (in red) for $\lambda=0.5$ (left), $\lambda=0.8$ (center) and $\lambda=1.2$ (right). The black trajectory on the top of the figure (right) is omitted in the other two plots for clarity, as it does not play any significant role in the analysis. }
\label{fig_trajectories}
\end{figure}

Figure \ref{fig_trajectories} illustrates the trajectories of the quadratic differential $-Q_{\lambda}(z)\dd z^2$ for different values of $\lambda<\lambda_0$.

\section{Proof of Theorem \ref{Th2}}\label{sec_Proof2}

In order to give the asymptotic behavior of the orthogonal polynomials $p^{\om}_n(z)$, we will use the Riemann--Hilbert approach, together with the Deift--Zhou nonlinear steepest descent method. We take the weight function $w_n(x)=\ee^{-n V(x)}$, with $V(x)=-\ii \lambda x$ and $0\leq\lambda<\lambda_0$,
where $\lambda_0$ is defined in \eqref{l0}.

\subsection{Riemann--Hilbert (RH) problem}\label{sec_RH}
 We seek $Y(z)=Y_{n,\omega}(z)\in\mathbb{C}^{2\times 2}$ such that
\begin{enumerate}
\item $Y(z)$ is analytic (entrywise) in $\mathbb{C}\setminus \gamma_{\lambda}$.
\item On $\gamma_{\lambda}$, we have the jump
\begin{equation}
 Y_+(x)=Y_-(x)\begin{pmatrix}1 & w_n(x)\\ 0 & 1\end{pmatrix}.
\end{equation}
\item As $z\to\infty$:
\begin{equation}
 Y(z)=\left(I+\mathcal{O}\left(\frac{1}{z}\right)\right)\begin{pmatrix} z^n & 0\\ 0 & z^{-n}\end{pmatrix}.
\end{equation}
\item As $z\to \pm 1$ we have
\begin{equation}
 Y(z)=\mathcal{O}\begin{pmatrix}
                   1 & \log|z\mp 1|\\
		   1 & \log|z\mp 1|
                  \end{pmatrix}.
\end{equation}
\end{enumerate}

It is known that if this RH problem has a solution, it is unique, see for instance \cite[\S 3.2 and \S 7.1]{Deift}. Moreover, in this case the solution is given by
\begin{equation}
 Y(z)=\begin{pmatrix}
       p^{\om}_n(z) & (\mathcal{C} p^{\om}_n w_n)(z)\\
       -2\pi\ii\kappa_{n-1}^2 p^{\om}_{n-1}(z) & -2\pi\ii \kappa_{n-1}^2(\mathcal{C} p^{\om}_{n-1}w_n)(z)
      \end{pmatrix},
\end{equation}
where
\begin{equation}
 (\mathcal{C} f)(z)=\frac{1}{2\pi\ii}\int_{\gamma_{\lambda}} \frac{f(s)}{s-z}\dd s
\end{equation}
is the Cauchy transform of the function $f(z)$, analytic in $\mathbb{C}\setminus \gamma_{\lambda}$, and $\kappa_{n-1}$ is the leading coefficient of the orthonormal polynomial, i.e. $\pi_{n-1}^{\omega}(x)=\kappa_{n-1}p_{n-1}^{\omega}(x)$.

The RH formulation of the orthogonal polynomials is due to Fokas, Its and Kitaev, \cite{FIK1}, see also the monograph of Deift, \cite[\S 3.2]{Deift}. The conditions at $z=\pm 1$ follow from a general result on the logarithmic behavior of the Cauchy transform, see \cite[\S 1.8.1]{Gak}.

The existence of $p^{\om}_n(z)$ and $p^{\om}_{n-1}(z)$ would be guaranteed for all $n$ if the weight function was positive, via the standard Gram--Schmidt orthogonalization procedure applied to the basis of monomials. In this case the weight function is not positive, so the existence of $p^{\om}_n(z)$ is not clear a priori, however the Deift--Zhou steepest descent analysis will provide a proof of existence for large enough $n$. This analysis consists of the following sequence of (explicit and invertible) transformations:
\begin{equation}
Y(z)\mapsto T(z) \mapsto S(z) \mapsto R(z).
\end{equation}

A global estimate of $R(z)$ for large $n$ will give the asymptotic behavior of $Y(z)$, and in particular of $p^{\om}_n(z)$ and $p^{\om}_{n-1}(z)$, and this will imply the existence of these polynomials, at least for large $n$.

\subsubsection{First transformation}
In order to normalize the RH problem at infinity, we need the $g$-function corresponding to the potential $V(z)$. This function is analytic in $\mathbb{C}\setminus((-\infty,-1]\cup\gamma_{\lambda})$, and on this contour it has the following jumps:
\begin{equation}
 g_+(x)-g_-(x)=\begin{cases}
                2\pi\ii\displaystyle \int_x^1 \psi_{\lambda}(s)\dd s,&\qquad x\in\gamma_{\lambda}\\
		2\pi\ii, \qquad x\in(-\infty,-1],
               \end{cases}
\end{equation}
where integration is taken along $\gamma_{\lambda}$, and $\psi_{\lambda}(s)$ is given by \eqref{psiz}. We consider the function
\begin{equation}\label{phi2}
 \phi(z)=2g(z)-V(z)-l,
\end{equation}
which is analytic in $\mathbb{C}\setminus((-\infty,-1]\cup\gamma_{\lambda})$ and satisfies 
\begin{equation}
\phi_+(x)=-\phi_-(x)=g_+(x)-g_-(x), \qquad x\in (-\infty,-1]\cup\gamma_{\lambda}.
\end{equation}

Using the analytic extension of $\psi(z)$, it is possible to write
\begin{equation}
 \phi(z)=2\pi\ii \int_z^1 \psi(s)\dd s,
\end{equation}
for $z\in\mathbb{C}\setminus((-\infty,-1]\cup\gamma_{\lambda})$. Integration is taken along the a smooth curve joining the points $z$ and $1$ in the complex plane, without crossing the cut $\gamma_{\lambda}$.

Direct evaluation of the integral gives $\phi(z)$ as in \eqref{phiz}:
\begin{equation}\label{phi2_explicit}
 \phi(z)=2\log\varphi(z)+\ii\lambda(z^2-1)^{1/2},
\end{equation}
again with the cut of the square root taken on $\gamma_{\lambda}$, and $\varphi(z)$ given by \eqref{varphi}.

Furthermore, we know from the integral representation that $g(z)=\log z+\ldots$ as $z\to\infty$, and
\begin{equation}
 g(z)=\frac{1}{2}\left(\phi(z)+V(z)+l\right)=\log 2+\log z+\frac{l}{2}+\mathcal{O}(1/z),
\end{equation}
using the explicit formulas for $\phi(z)$ and $V(z)$, so $l=-2\log 2$. The function $g(z)$ can be given explicitly as well:
\begin{equation}\label{gz}
  g(z)=\log\varphi(z)+\frac{\ii\lambda}{2}(z^2-1)^{1/2}-\frac{\ii\lambda z}{2}-\log 2,
\end{equation}
and it is analytic in $\mathbb{C}\setminus (-\infty,-1]\cup\gamma_{\lambda})$.

The first transformation reads
\begin{equation}\label{TY}
 T(z)=\ee^{-\frac{nl}{2}\sigma_3} Y(z)\ee^{-n\left[g(z)-\frac{l}{2}\right]\sigma_3}=2^{n\sigma_3} Y(z)\ee^{-n\left[g(z)-\frac{l}{2}\right]\sigma_3},
\end{equation}
where we have use the standard notation for the Pauli matrix $\sigma_3=\begin{pmatrix} 1&0\\ 0 & -1 \end{pmatrix}$.
Note that when $\lambda=0$, we get
\begin{equation}
 g(z)-\frac{l}{2}=g(z)+\log 2=\log\varphi(z),
\end{equation}
so
\begin{equation}
 \ee^{-n\left[g(z)-\frac{l}{2}\right]}=\varphi(z)^{-n},
\end{equation}
which coincides with the definition in \cite{KMcLVAV}.

Then the matrix $T(z)$ satisfies the following RH problem:
\begin{enumerate}
\item $T(z)$ is analytic (entrywise) in $\mathbb{C}\setminus\gamma_{\lambda}$
\item On $\gamma_{\lambda}$, we have the jump
\begin{equation}
 T_+(z)=T_-(z)
 \begin{pmatrix}
  \ee^{-n\phi_+(z)} & 1\\ 0 & \ee^{n\phi_+(z)}
 \end{pmatrix},
\qquad z\in\gamma_{\lambda}. 
 \end{equation}
\item As $z\to\infty$:
\begin{equation}
 T(z)=I+\mathcal{O}\left(\frac{1}{z}\right)
\end{equation}
\item As $z\to \pm 1$, $T(z)$ has the same behavior as $Y(z)$. 

\end{enumerate}

\subsubsection{Second transformation}
We note the following factorization of the jump matrix on $\gamma_{\lambda}$:
\begin{equation}
\begin{pmatrix}
  \ee^{-n\phi_+(z)} & 1\\ 0 & \ee^{n\phi_+(z)}
 \end{pmatrix}
=
\begin{pmatrix}
  1 & 0\\ \ee^{n\phi_+(z)} & 1
 \end{pmatrix}
\begin{pmatrix}
  0 & 1\\ -1 & 0
\end{pmatrix}
\begin{pmatrix}
  1 & 0\\ \ee^{-n\phi_+(z)} & 1
 \end{pmatrix}
\end{equation}

In the next transformation we open a lens around $\gamma_{\lambda}$, see Figure \ref{fig_lens}, and we consider the following matrix function:
\begin{equation}\label{ST}
S(z)=\begin{cases}
T(z), \qquad z\,\,\textrm{outside the lens}\\
T(z)\begin{pmatrix} 1& 0\\ -\ee^{-n\phi(z)} & 1\end{pmatrix}, \qquad z\,\,\textrm{in the upper part of the lens}\\
T(z)\begin{pmatrix} 1& 0\\ \ee^{-n\phi(z)} & 1\end{pmatrix}, \qquad z\,\,\textrm{in the lower part of the lens}.
\end{cases}
\end{equation}

\begin{figure}[h]
\centerline{\includegraphics[width=85mm,height=40mm]{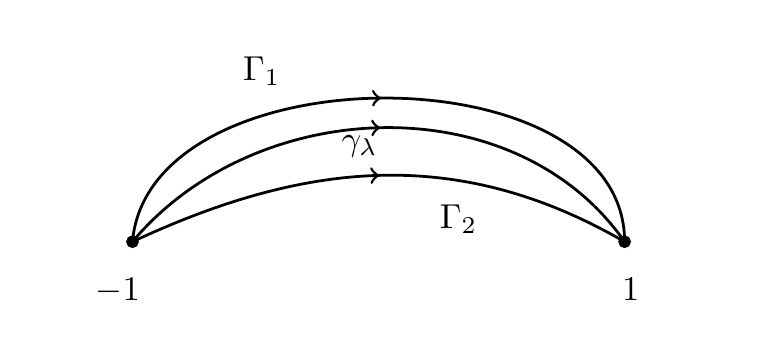}}
\caption{Lens--shaped contour $\Sigma_S$ around the curve $\gamma_{\lambda}$.}
\label{fig_lens}
\end{figure}

The matrix $S(z)$ satisfies the following RH problem:
\begin{enumerate}
\item $S(z)$ is analytic (entrywise) in $\mathbb{C}\setminus\Sigma_S$, where $\Sigma_S=\Gamma_1\cup\gamma_{\lambda}\cup\Gamma_2$, see Figure \ref{fig_lens}
\item On $\Sigma_S\setminus\{-1,1\}$ we have the following jumps:
\begin{equation}
 S_+(z)=S_-(z)
\begin{cases}
\begin{pmatrix}1 & 0\\ \ee^{-n\phi(z)} & 1\end{pmatrix}, \qquad z\in \Gamma_1\cup\Gamma_2,\\
\begin{pmatrix}0 & 1\\ -1 & 0\end{pmatrix}, \qquad z\in \gamma_{\lambda}. 
\end{cases}
\end{equation}
\item As $z\to\infty$:
\begin{equation}
 S(z)=I+\mathcal{O}\left(\frac{1}{z}\right).
\end{equation}
\item As $z\to 1$ we have
\begin{equation}
 S(z)=
       \begin{cases}
        \mathcal{O}\begin{pmatrix}
                    1 & \log|z-1|\\
 		   1 & \log|z-1|
                   \end{pmatrix},
       &z\,\,\textrm{outside the lens}\\
        \mathcal{O}\begin{pmatrix}
                   \log|z-1| & \log|z-1|\\
		   \log|z-1| & \log|z-1|
                  \end{pmatrix},
       &z\,\,\textrm{inside the lens}
       \end{cases}
\end{equation}
\item As $z\to -1$ we have the same behavior as before, replacing $|z-1|$ with $|z+1|$.
\end{enumerate}

Since the function $\phi(z)$ is purely imaginary on $\gamma_{\lambda}$ and $\textrm{Im}\,\phi(z)$ is decreasing along $\gamma_{\lambda}$, the Cauchy--Riemann equations imply that $\textrm{Re}\, \phi(z)$ increases locally in the complex plane as we move up in the complex plane from the curve $\gamma_{\lambda}$. Hence the jump matrix on $\Gamma_1$ tends to the identity exponentially fast with $n$. A similar reasoning applies to the jump on the curve $\Gamma_2$.

\subsubsection{Model RH problem}\label{modelRH}
Ignoring all jumps that are exponentially close to the identity as $n\to\infty$, we seek a matrix $N(z)$ that satisfies the following RH problem:
\begin{enumerate}
\item $N(z)$ is analytic in $\mathbb{C}\setminus \gamma_{\lambda}$.
\item On $\gamma_{\lambda}$, we have the following jump:
\begin{equation}
 N_+(x)=N_-(x)
\begin{pmatrix} 0& 1\\ -1 & 0\end{pmatrix}, 
\end{equation}
\item As $z\to\infty$:
\begin{equation}
 N(z)=I+\mathcal{O}\left(\frac{1}{z}\right).
\end{equation}
\end{enumerate}

This RH problem can be solved explicitly, and we have
\begin{equation}\label{Nz}
N(z)=\frac{1}{\sqrt{2}(z^2-1)^{1/4}}
\begin{pmatrix} 
\varphi(z)^{1/2} & \ii\varphi(z)^{-1/2}\\
-\ii \varphi(z)^{-1/2} & \varphi(z)^{-1/2}
\end{pmatrix},
\end{equation}
in terms of the function $\varphi(z)=z+(z^2-1)^{1/2}$, taken analytic in $\mathbb{C}\setminus\gamma_{\lambda}$ .

\subsubsection{Local parametrices}
We consider a disc $D(1,\delta)$ around $z=1$, with radius $\delta>0$ fixed. We seek a matrix $P$ that satisfies the following RH problem: 

\begin{enumerate}
\item $P(z)$ is analytic in $D(1,\delta)\setminus\Sigma_S$.
\item On $D(1,\delta)\cap \Sigma_S$, we have the following jumps:
\begin{equation}
 P_+(x)=P_-(x)
\begin{cases}
\begin{pmatrix}1 & 0\\ \ee^{-n\phi(z)} & 1\end{pmatrix}, \qquad z\in D(1,\delta)\cap(\Gamma_1\cup\Gamma_2),\\
\begin{pmatrix}0 & 1\\ -1 & 0\end{pmatrix}, \qquad z\in D(1,\delta)\cap \gamma_{\lambda}. 
\end{cases}
\end{equation}
\item Uniformly for $z\in\partial D(1,\delta)$, we have the matching
\begin{equation}\label{matchingPN}
 P(z)=N(z)\left(I+\mathcal{O}\left(\frac{1}{n}\right)\right), \qquad n\to\infty.
\end{equation}
\item As $z\to 1$ we have
\begin{equation}
 P(z)=\begin{cases}
        \mathcal{O}\begin{pmatrix}
                    1 & \log|z-1|\\
 		   1 & \log|z-1|
                   \end{pmatrix},
       &z\,\,\textrm{outside the lens}\\
        \mathcal{O}\begin{pmatrix}
                   \log|z-1| & \log|z-1|\\
		   \log|z-1| & \log|z-1|
                  \end{pmatrix},
       &z\,\,\textrm{inside the lens}
\end{cases}
\end{equation}
\end{enumerate}

We construct the local parametrix in the following way: define
\begin{equation} \label{Pz}
 P(z)=E_n(z)Q(z)\ee^{-\frac{n}{2}\phi(z)\sigma_3},
\end{equation}
where $E_n(z)$ is an analytic factor, to be determined later to get the matching condition with $N(z)$ on the boundary of the disc. Then $Q(z)$ satisfies the following RH problem:
\begin{enumerate}
\item $Q(z)$ is analytic in $D(1,\delta)\setminus\Sigma_S$.
\item On $D(1,\delta)\cap \Sigma_S$, we have the following jumps:
\begin{equation}
 Q_+(z)=Q_-(z)
\begin{cases}
\begin{pmatrix}1 & 0\\ 1 & 1\end{pmatrix}, \qquad z\in D(1,\delta)\cap(\Gamma_1\cup\Gamma_2),\\
\begin{pmatrix}0 & 1\\ -1 & 0\end{pmatrix}, \qquad z\in D(1,\delta)\cap \gamma_{\lambda}. 
\end{cases}
\end{equation}
\item As $z\to 1$ we have
\begin{equation}
 Q(z)=\begin{cases}
        \mathcal{O}\begin{pmatrix}
                    1 & \log|z-1|\\
 		   1 & \log|z-1|
                   \end{pmatrix},
       &z\,\,\textrm{outside the lens}\\
        \mathcal{O}\begin{pmatrix}
                   \log|z-1| & \log|z-1|\\
		   \log|z-1| & \log|z-1|
                  \end{pmatrix},
       &z\,\,\textrm{inside the lens}
\end{cases}
\end{equation}
\end{enumerate}

The matrix $Q(z)$ is constructed using Bessel functions, following the ideas in \cite[Section 6]{KMcLVAV}.
\begin{equation}
 Q(z)=\Psi(n^2 f(z)),
\end{equation}
where $\zeta=f(z)$ is a conformal mapping from a neighborhood of $z=1$ onto a neighborhood of the origin, and $\Psi(\z)$ solves the following auxiliary RH problem: in the $\z$ plane, consider the contour $\Sigma_{\Psi}$ as illustrated in Figure \ref{fig_mapping_Bessel}, then 
\begin{enumerate}
\item $\Psi(\z)$ is analytic in $\mathbb{C}\setminus\Sigma_{\Psi}$.
\item On $\Sigma_{\Psi}$, we have the following jumps:
\begin{equation}
 \Psi_+(\z)=\Psi_-(\z)
\begin{cases}
\begin{pmatrix}1 & 0\\ 1 & 1\end{pmatrix}, \qquad \z\in \Sigma_1\cup\Sigma_3,\\
\begin{pmatrix}0 & 1\\ -1 & 0\end{pmatrix}, \qquad \z\in\Sigma_2. 
\end{cases}
\end{equation}
\item As $\z\to 0$ we have the following behavior:
\begin{equation}
 \Psi(\z)=\begin{cases}
        \mathcal{O}\begin{pmatrix}
                    1 & \log|\z|\\
 		   1 & \log|\z|
                   \end{pmatrix},
       &|\arg\zeta|<\frac{2\pi}{3}\\
        \mathcal{O}\begin{pmatrix}
                   \log|\zeta| & \log|\zeta|\\
		   \log|\zeta| & \log|\zeta|
                  \end{pmatrix},
       &\textrm{elsewhere}
\end{cases}
\end{equation}
\end{enumerate}

\begin{figure}
 \centerline{\includegraphics{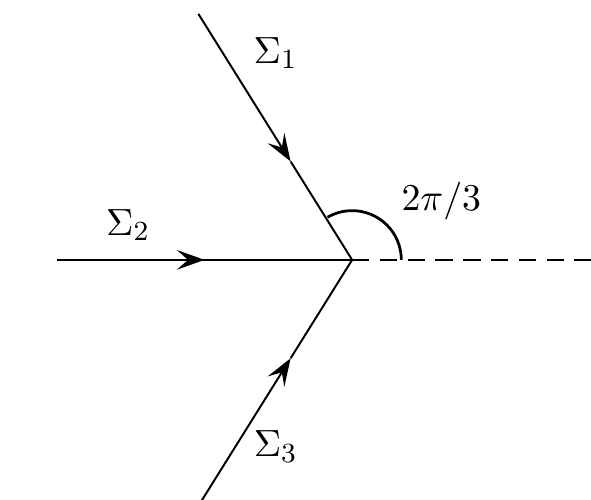}}
\caption{Contour $\Sigma_{\Psi}$ in the $\z$ plane}
\label{fig_mapping_Bessel}
\end{figure}

The solution of this RH problem is given in terms of modified and ordinary Bessel functions, see \cite[Section 6]{KMcLVAV}:
\begin{equation}
 \Psi(\z)=
 \begin{pmatrix}
  \frac{1}{2}H_0^{(2)}(2(-\z)^{1/2}) &-\frac{1}{2}H_0^{(1)}(2(-\z)^{1/2})\\[1mm]
  -\pi \z^{1/2} (H_0^{(2)})'(2(-\z)^{1/2}) & \pi \z^{1/2} (H_0^{(1)})'(2(-\z)^{1/2})
 \end{pmatrix}
\end{equation}
if $-\pi<\arg \z<-\frac{2\pi}{3}$,
\begin{equation}
 \Psi(\z)=
 \begin{pmatrix}
  I_0(2\z^{1/2}) &\frac{i}{\pi}K_0(2\z^{1/2})\\[1mm]
  2\pi \ii \z^{1/2} I_0'(2\z^{1/2}) & -2\z^{1/2} K_0'(2\z^{1/2})
 \end{pmatrix}
\end{equation}
if $|\arg \z|<\frac{2\pi}{3}$ and
\begin{equation}
 \Psi(\z)=
 \begin{pmatrix}
  \frac{1}{2}H_0^{(1)}(2(-\z)^{1/2}) & \frac{1}{2}H_0^{(2)}(2(-\z)^{1/2})\\[1mm]
  \pi  \z^{1/2} (H_0^{(1)})'(2(-\z)^{1/2}) & \pi \z^{1/2} (H_0^{(2)})'(2(-\z)^{1/2})
 \end{pmatrix}
\end{equation}
if $\frac{2\pi}{3}<\arg \z<\pi$. Consider the asymptotic behavior of this $\Psi(\z)$ function. If $|\arg \z|<2\pi/3$, we have
\begin{equation}
 \Psi(\z)=\frac{1}{\sqrt{2}}(2\pi\z^{1/2})^{-\sigma_3/2}
\begin{pmatrix}
 1 & \ii\\
 \ii & 1
\end{pmatrix}
\left(I+\mathcal{O}(\z^{-1/2})\right)
\ee^{2\z^{1/2}\sigma_3}, \qquad \z\to\infty,
\end{equation}
where the square root takes the principal value, with a cut on the negative real axis. The same result can be checked in the other two sectors, using known connection formulas for Bessel functions, and the behavior of $\Psi(\zeta)$ at the origin also follows from standard expansions of Hankel and Bessel functions.

It remains to determine the conformal mapping $f(z)$ and the analytic prefactor $E_n(z)$. First, we want to match the exponential factors in $Q(z)$ and $\Psi(z)$, so we set
\begin{equation}
 \z=n^2f(z)=\frac{n^2}{16}\phi(z)^2.
\end{equation}

Hence, $f(z)=\phi(z)^2/16$, and since
\begin{equation}
 \phi(z)=\sqrt{2}\left(2+\ii\lambda\right)(z-1)^{1/2}+\mathcal{O}((z-1)^{3/2}), \qquad z\to 1,
\end{equation}
we obtain
\begin{equation}
 f(z)=\frac{(2+\ii\lambda)^2}{8}(z-1)+\mathcal{O}((z-1)^2), \qquad z\to 1,
\end{equation}
so $f(z)$ is indeed a conformal mapping locally near $z=1$. Moreover,
\begin{equation}
 \Psi(n^2 f(z))=\frac{1}{\sqrt{2}}(2\pi n)^{-\sigma_3/2} f(z)^{-\sigma_3/4}
\begin{pmatrix}
 1 & \ii\\
 \ii & 1
\end{pmatrix}
\left(I+\mathcal{O}(n^{-1})\right)
\ee^{2nf(z)^{1/2}\sigma_3},
\end{equation}
so we define
\begin{equation}
 E_n(z)=N(z)\frac{1}{\sqrt{2}}\begin{pmatrix}
                                                              1 & -\ii\\
							       -\ii & 1
                                                             \end{pmatrix}
f(z)^{\sigma_3/4}(2\pi n)^{\sigma_3/2},
\end{equation}
and thus the matching of $P(z)$ and $N(z)$ required in \eqref{matchingPN} is achieved on the boundary of $D(1,\delta)$.

In a neighborhood of $z=-1$, say $D(-1,\delta)$, we have a similar construction: the local parametrix $\tilde{P}(z)$ is given by
\begin{equation}
 \tilde{P}(z)=\tilde{E}_n(z)\tilde{Q}(z)\ee^{-\frac{n}{2}\tilde{\phi}(z)\sigma_3},
\end{equation}
where 
\begin{equation}
\tilde{\phi}(z)=2\pi\ii\int_z^{-1} \psi(s)\dd s=\phi(z)-2\pi\ii
\end{equation}
and
\begin{equation}
 \tilde{Q}(z)=\tilde{\Psi}(n^2 \tilde{f}(z))
=\sigma_3\Psi(n^2\tilde{f}(z))\sigma_3,
\end{equation}
with $\tilde{f}(z)=\tilde{\phi}(z)^2/16$ a conformal mapping from a neighborhood of $z=-1$ onto a 
neighborhood of $\z=0$. The analytic factor in this case is
\begin{equation}
 \tilde{E}_n(z)=N(z)\frac{1}{\sqrt{2}}\begin{pmatrix}
                                      1 & \ii\\
				      \ii & 1
                                     \end{pmatrix}
\tilde{f}(z)^{\sigma_3/4}(2\pi n)^{\sigma_3/2}.
\end{equation}

\subsubsection{Final transformation}
Using the global and local parametrices, we define
\begin{equation}\label{RS}
 R(z)=S(z)
\begin{cases}
 N^{-1}(z), & \qquad z\in\mathbb{C}\setminus (D(1,\delta)\cup D(-1,\delta)\cup\Sigma_S),\\
P^{-1}(z), & \qquad z\in D(1,\delta)\setminus\Sigma_S,\\
\tilde{P}^{-1}(z), & \qquad z\in D(-1,\delta)\setminus\Sigma_S,
\end{cases}
\end{equation}

\begin{figure}
\centerline{\includegraphics[width=90mm,height=45mm]{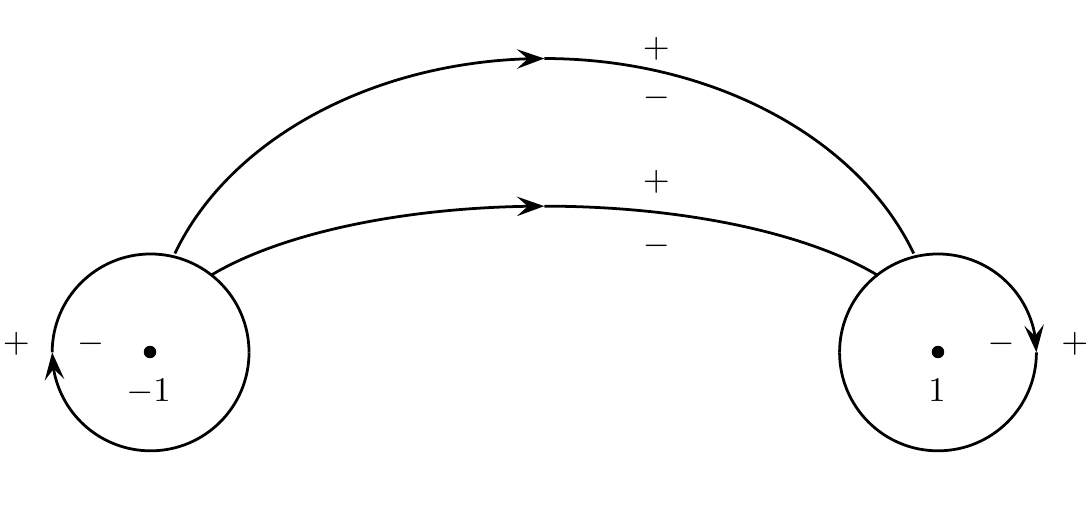}}
\caption{Final contour $\Sigma_R$ in the steepest descent analysis}
\label{fig_SigmaR}
\end{figure}

Following an argument similar to the one in \cite[\S 7]{KMcLVAV}, it can be shown that this matrix $R(z)$ is analytic off the contour $\Sigma_R$ shown in Figure \ref{fig_SigmaR}. Furthermore, as $n\to\infty$ the jumps are
\begin{equation}
 R_+(z)=R_-(z)
\begin{cases}
 I+\mathcal{O}(\ee^{-2cn}), \qquad z\in\Sigma_R\setminus(\partial D(1,\delta)\cup\partial D(-1,\delta)),\\
 I+\mathcal{O}(n^{-1}), \qquad z\in\partial D(1,\delta)\cup\partial D(-1,\delta).
\end{cases}
\end{equation}

From this result and the fact that $R(z)\to I$ as $z\to\infty$, it can be proved that
\begin{equation}
 R(z)=I+\mathcal{O}(n^{-1}), \qquad n\to\infty,
\end{equation}
uniformly for $z\in\mathbb{C}\setminus\Sigma_R$, see for example \cite[Section 11]{Ble}. Reversing the transformations, we can obtain the asymptotic behavior of $Y(z)$ as $n\to\infty$, and in particular of the $(1,1)$ entry, that contains the polynomial $p^{\om}_n(z)$. Actually, the matrix $R(z)$ admits a full asymptotic expansion in inverse powers of $n$:
\begin{equation}\label{asympRn}
 R(z)\sim I+\sum_{k=1}^{\infty}\frac{R^{(k)}(z)}{n^k}, \qquad n\to\infty,
\end{equation}
where the functions $R^{(k)}(z)$ are analytic in $z\in\mathbb{C}\setminus(\partial D(1,\delta)\cup\partial D(-1,\delta))$. Furthermore, it is true that
\begin{equation}
R^{(k)}(z)=\mathcal{O}\left(\frac{1}{z}\right), \qquad z\to\infty,
\end{equation}
and, as in \cite[Lemma 8.3]{KMcLVAV}, the expansion \eqref{asympRn} is uniformly valid for large values of $z$ as well. This double asymptotic property, for large $n$ and large $z$, will be used later to obtain the asymptotic expansion of the recurrence coefficients.

The outer asymptotics for $p_n^{\om}(z)$, for $z$ in compact subsets of $\mathbb{C}\setminus\gamma_{\lambda}$ (outside of the lens and away from the endpoints), can be worked out from the steepest descent analysis. We apply the first transformation \eqref{TY}, together with the fact that in the outer region $T(z)=S(z)=R(z)N(z)$, so
\begin{equation}
\begin{aligned}
 p^{\om}_n(z)=\begin{pmatrix} 1 & 0 \end{pmatrix} Y(z)\begin{pmatrix} 1 \\ 0 \end{pmatrix}
&=\ee^{ng(z)}\begin{pmatrix} 1 & 0 \end{pmatrix} R(z)N(z)
\begin{pmatrix} 1 \\ 0 \end{pmatrix}\\
&=\ee^{ng(z)}\left(R_{11}(z)N_{11}(z)+R_{12}(z)N_{21}(z)\right)
\end{aligned}
\end{equation}

Now we use the fact that $R_{11}(z)=1+\mathcal{O}(1/n)$ and $R_{12}(z)=\mathcal{O}(1/n)$, and the explicit form of $N_{11}(z)$, see \eqref{Nz}. Finally, note that
\[
\ee^{ng(z)}=\left(\frac{\varphi(z)}{2}\right)^n \exp\left(-\frac{\ii n\lambda}{2\varphi(z)}\right)
\]
in terms of the function $\varphi(z)=z+(z^2-1)^{1/2}$, and using the explicit formula for $g(z)$, see \eqref{gz} and $D(z)$. Thus we get the asymptotic expansion  \eqref{outer_pn}.

If $z$ is in the upper (lower) part of the lens, we get
\begin{equation}
\begin{aligned}
 p^{\om}_n(z)&=\begin{pmatrix} 1 & 0 \end{pmatrix} Y(z)\begin{pmatrix} 1 \\ 0 \end{pmatrix}\\
 & = \begin{pmatrix} \ee^{\frac{nl}{2}} & 0 \end{pmatrix} S(z)
 \begin{pmatrix} 1 & 0 \\ \pm \ee^{-n\phi(z)} & 1 \end{pmatrix} 
\begin{pmatrix} \ee^{n\left(g(z)-\frac{l}{2}\right)}\\ 0 \end{pmatrix}\\
 & = \begin{pmatrix} \ee^{\frac{nl}{2}} & 0 \end{pmatrix} R(z)N(z)
\begin{pmatrix} \ee^{n\left(g(z)-\frac{l}{2}\right)}\\ \pm\ee^{n\left(g(z)-\phi(z)-\frac{l}{2}\right)} \end{pmatrix}\\
 \end{aligned}
\end{equation}

Since 
$$
g(z)-\frac{l}{2}=\frac{V(z)}{2}-\frac{\phi(z)}{2},
$$
we obtain
$$
\begin{aligned}
 &p^{\om}_n(z) = \ee^{\frac{n}{2}(V(z)+l)}\begin{pmatrix} 1 & 0 \end{pmatrix} R(z)N(z)
\begin{pmatrix} \ee^{\frac{n}{2}\phi(z)}\\ \pm \ee^{-\frac{n}{2}\phi(z)}\end{pmatrix}\\
&=\ee^{\frac{n}{2}(V(z)+l)}\left(N_{11}(z)\ee^{\frac{n\phi(z)}{2}}
\pm N_{12}(z)\ee^{-\frac{n\phi(z)}{2}}+\mathcal{O}(1/n)\right),
 \end{aligned}
$$
bearing in mind again that $R_{11}(z)=1+\mathcal{O}(1/n)$ and $R_{12}(z)=\mathcal{O}(1/n)$. Observe that if we take boundary values on $\gamma_{\lambda}$, we can use the fact that $\phi_+(z)=-\phi_-(z)$, and also $N_{11+}(z)=-N_{12-}(z)$ and $N_{12+}(z)=N_{11-}(z)$, which come from the model Riemann--Hilbert problem. As a result, both boundary values coincide on the curve $\gamma_{\lambda}$. 

Now we combine this with the explicit form of $\phi(z)$, see \eqref{phi2_explicit}, and with the formulas for $N_{11}(z)$ and $N_{12}(z)$. Note also that on $\gamma_{\lambda}$ we can write $\varphi_+(z)=\ee^{\ii\arccos z}$, using the standard definition of the arccosine function, with a cut on $(-\infty,-1]\cup[1,\infty)$, see \cite[\S 4.23.22]{DLMF}. Thus, substituing the value of $l$ and simplifying, we obtain the asymptotic result \eqref{inner_pn} on a neighborhood of the curve $\gamma_{\lambda}$. 

For $z\in D(1,\delta)$ and in the upper part of the lens, we use the relation between $T(z)$ and $S(z)$ in \eqref{ST}, the connection between $S(z)$ and $R(z)$ in \eqref{RS}, and the the expression for $P(z)$ in \eqref{Pz} to write
\begin{equation}
\begin{aligned}
 p^{\om}_n(z)&=\ee^{\frac{nl}{2}}
\begin{pmatrix} 1 & 0 \end{pmatrix} R(z)P(z)\begin{pmatrix} \ee^{n\left(g(z)-\frac{l}{2}\right)}\\ \ee^{n\left(g(z)-\frac{l}{2}-\phi(z)\right)} \end{pmatrix}\\
 &=\ee^{\frac{n}{2}(V(z)+l)}\begin{pmatrix} 1 & 0 \end{pmatrix}\left(I+\mathcal{O}\left(\frac{1}{n}\right)\right) E_n(z)\Psi(n^2 f(z))
 \begin{pmatrix} 1\\ 1 \end{pmatrix}
 \end{aligned}
\end{equation}

Now we observe that
\begin{equation}
\begin{aligned}
E_n(z)&=\frac{1}{\sqrt{2}}N(z)\begin{pmatrix}
1& -\ii\\ -\ii & 1
\end{pmatrix}
f(z)^{\sigma_3/4} (2\pi n)^{\sigma_3/2}\\
&=\frac{1}{\sqrt{2}}
\begin{pmatrix}
\beta(z)^{-1} & -\ii\beta(z)\\
-\ii\beta(z)^{-1} & \beta(z)\\
\end{pmatrix}
f(z)^{\sigma_3/4} (2\pi n)^{\sigma_3/2},
\end{aligned}
\end{equation}
in terms of $\beta(z)=\left(\frac{z-1}{z+1}\right)^{1/4}$.

Also, bearing in mind that $f(z)=\phi^2(z)/16$, we have
\[
 \Psi(n^2 f(z)) \begin{pmatrix} 1 \\ 1 \end{pmatrix}=
 \begin{pmatrix} 
  J_0\left(2n(-f(z))^{1/2}\right) \\ \pi n f(z)^{1/2} (J_0)'\left(2n(-f(z))^{1/2}\right) 
 \end{pmatrix},
\]
using the standard connection between Hankel and Bessel functions, see for instance \cite[10.4.3]{DLMF}. As a consequence,
\[
 f(z)^{\sigma_3/4} (2\pi n)^{\sigma_3/2} \Psi(n^2 f(z)) \begin{pmatrix} 1 \\ 1 \end{pmatrix}=
 f(z)^{1/4} (2\pi n)^{1/2}
 \begin{pmatrix} 
  J_0\left(2n(-f(z))^{1/2}\right) \\ (J_0)'\left(2n(-f(z))^{1/2}\right) 
 \end{pmatrix},
\]
and putting all together we get 
\[
\begin{aligned}
p^{\om}_n(z)&=\ee^{\frac{n}{2}(V(z)+l)}f(z)^{1/4} (\pi n)^{1/2}
\begin{pmatrix} 1 & 0 \end{pmatrix}\left(I+\mathcal{O}\left(\frac{1}{n}\right)\right)\\
&\times
\begin{pmatrix}
\beta(z)^{-1} & -\ii\beta(z)\\
-\ii\beta(z)^{-1} & \beta(z)
\end{pmatrix}
\begin{pmatrix} 
  J_0\left(2n(-f(z))^{1/2}\right) \\ (J_0)'\left(2n(-f(z))^{1/2}\right) 
 \end{pmatrix} 
\end{aligned}
\]

Finally, the argument of the Bessel function is equal to $-\frac{\ii n}{2}\phi(z)$, using the relation between $f(z)$ and $\phi(z)$. The sign of the square root is determined by the fact that in that sector of the lens we have
$\frac{2\pi}{3}<\arg\phi(z)<\pi$. Substituting $l=-2\log 2$, we obtain the asymptotic expansion \eqref{right_pn}. A similar computation can be carried out in the disc around $z=-1$, but we omit it for brevity.

It is clear that using the properties of $R(z)$, in particular its asymptotic expansion for large $n$, it is possible to compute higher order terms in the asymptotic expansion for the orthogonal polynomials if one can obtain the terms $R^{(k)}(z)$ in \eqref{asympRn}. Thus, substitution of a more refined asymptotic estimate for $R(z)$ in the previous formulas will lead to more accurate estimates. We omit this in the present paper for brevity, but we will present a possible approach in the next section, taken from \cite{KMcLVAV} and that is exploited in \cite{DHO}.

\section{Proof of Theorem \ref{Th3}}\label{sec_Proof3}

The recurrence coefficients of the three term recursion can be written in terms of the matrices in the Riemann--Hilbert analysis:
\begin{equation}
\begin{aligned}
a^2_n&=[Y_1]_{12}[Y_1]_{21},\\ 
b_n &=\frac{[Y_2]_{12}}{[Y_1]_{12}}-[Y_1]_{22},
\end{aligned}
\end{equation}
where
\begin{equation}
Y(z) z^{-n\sigma_3}=I+\frac{Y_1}{z}+\frac{Y_2}{z^2}+\ldots,\qquad z\to\infty.
\end{equation}

Hence, the solvability of the Riemann--Hilbert problem for large $n$ proves the existence of the recurrence coefficients for large $n$. In order to obtain the asymptotic estimates, we want to write the coefficients in terms of $R(z)$, so we denote
\begin{equation}
T(z)=I+\frac{T_1}{z}+\frac{T_2}{z^2}+\ldots, \qquad z\to\infty.
\end{equation}

We use the fact that
\begin{equation}
 g(z)=\log z-\frac{c_1}{z}-\frac{c_2}{z^2}+\mathcal{O}(z^{-3}), \qquad z\to\infty,
\end{equation}
where the constants $c_1$ and $c_2$ are given by
\begin{equation}
c_1=\int_{\gamma_{\lambda}} s\,\dd\mu_{\lambda}(s), \qquad
c_2=\int_{\gamma_{\lambda}} s^2\,\dd\mu_{\lambda}(s).
\end{equation}

These constants can be computed explicitly in this case, using residue calculus, but this is not really needed. Consequently,
\begin{equation}
 \ee^{-ng(z)}=z^{-n}\left(1+\frac{nc_1}{z}+\frac{2n^2c_1^2+nc_2}{2z^2}+\mathcal{O}(z^{-3})\right), \qquad z\to\infty,
\end{equation}

 Next, because of the relation between $Y(z)$ and $T(z)$, see \eqref{TY}, and expanding for large $z$, we have
\begin{equation}
\begin{aligned}
T(z)=\ee^{-\frac{nl}{2}\sigma_3}&\left(I+\frac{Y_1}{z}+\frac{Y_2}{z^2}+\ldots\right)\times\\
&\left(I+\frac{nc_1\sigma_3}{z}+
\frac{(2nc_1^2+c_2)n\sigma_3}{2z^2}+\ldots\right)\ee^{\frac{nl}{2}\sigma_3}, 
\end{aligned}
\end{equation}
so
\begin{equation}
\begin{aligned}
T_1&=\ee^{-\frac{nl}{2}\sigma_3}Y_1\ee^{\frac{nl}{2}\sigma_3}+c_1n\sigma_3,\\
T_2&=\ee^{-\frac{nl}{2}\sigma_3}Y_2\ee^{\frac{nl}{2}\sigma_3}+c_1 n\ee^{-\frac{nl}{2}\sigma_3}Y_1\ee^{\frac{nl}{2}\sigma_3}+
\frac{(2nc_1^2+c_2)n\sigma_3}{2}
\end{aligned}
\end{equation}

It follows that 
\begin{equation}
[T_1]_{12}=\ee^{-nl}[Y_1]_{12}, \qquad
[T_1]_{21}=\ee^{nl}[Y_1]_{21}, \qquad
[Y_1]_{22}=[T_1]_{22}-c_1n,
\end{equation}
and also
\begin{equation}
 [T_2]_{12}=\ee^{-nl}\left(-c_1n[Y_1]_{12}+[Y_2]_{12}\right).
\end{equation}

As a consequence, we have in terms of $T(z)$:
\begin{equation}
\begin{aligned}
a^2_n&=[T_1]_{12}[T_1]_{21},\\
b_n &=\frac{[T_2]_{12}}{[T_1]_{12}}-[T_1]_{22},
\end{aligned}
\end{equation}

Next, away from the curve $\gamma_{\lambda}$, we have $T(z)=S(z)=R(z)N(z)$, and we define
\begin{equation}
N(z)=I+\frac{N_1}{z}+\frac{N_2}{z^2}+\ldots, \qquad
R(z)=I+\frac{R_1}{z}+\frac{R_2}{z^2}+\ldots, 
\end{equation}
as $z\to\infty$. Then,
\begin{equation}
\begin{aligned}
 T_1&=R_1+N_1,\\
 T_2&=N_2+R_1 N_1+R_2.
\end{aligned}
\end{equation}

From \eqref{Nz}, we have
\begin{equation}
N(z)=I-\frac{1}{2z}\sigma_2+\frac{1}{8z^2}I+\ldots, \qquad
\sigma_2=\begin{pmatrix} 0 & -\ii\\ \ii & 0 \end{pmatrix},
\end{equation}
so
\begin{equation}
\begin{aligned}
 T_1&=R_1-\frac{1}{2}\sigma_2,\\
 T_2&=\frac{1}{8}I-\frac{1}{2}R_1 \sigma_2+R_2,
\end{aligned}
\end{equation}
and then
\begin{equation}\label{anbnR}
\begin{aligned}
a^2_n&=\left([R_1]_{12}+\frac{\ii}{2}\right)\left([R_1]_{21}-\frac{\ii}{2}\right),\\
b_n &=\frac{\ii[R_1]_{11}+2[R_2]_{12}}{\ii+2[R_1]_{12}}-[R_1]_{22}.
\end{aligned}
\end{equation}

In order to compute the terms $R_1$ and $R_2$, we need to obtain additional terms in the large $n$ asymptotics of $R(z)$ first, that is, the coefficients $R^{(k)}(z)$ in \eqref{asympRn}. Following \cite[\S 8]{KMcLVAV}, we write the jump matrix for $R(z)$ on $\Sigma_R$ as a perturbation of the identity, so
\begin{equation}
R_+(z)=R_-(z)\left(I+\Delta(z)\right), \qquad z\in\Sigma_R,
\end{equation}
where $\Delta(z)$ admits an expansion in inverse powers of $n$:
\begin{equation}
\Delta(z)\sim\sum_{k=1}\frac{\Delta_k(z)}{n^k}, \qquad n\to\infty.
\end{equation}

Since the jump is exponentially close to the identity on $\Sigma_R\setminus(\partial D(1,\delta)\cup \partial D(-1,\delta))$, the coefficients $\Delta_k(z)$ are identically $0$ there. For $z\in\partial D(1,\delta)\cup \partial D(-1,\delta)$, they can be computed from the asymptotic expansion of the Bessel functions in the local parametrices. The outcome is the following:
\begin{equation}\label{Deltak1}
\begin{aligned}
\Delta_k(z)&=\frac{(-1)^{k-1}}{4^{k-1} (k-1)! \phi(z)^k} \prod_{j=1}^{k-1} (2j-1)^2\times\\
&N(z)
\begin{pmatrix}
\frac{(-1)^k}{k}\left(\frac{k}{2}-\frac{1}{4}\right) & -\left(k-\frac{1}{2}\right)\ii\\[2mm]
(-1)^k \left(k-\frac{1}{2}\right)\ii  & \frac{1}{k}\left(\frac{k}{2}-\frac{1}{4}\right) & 
\end{pmatrix}
N^{-1}(z), \qquad z\in\partial D(1,\delta)
\end{aligned}
\end{equation}
and
\begin{equation}\label{Deltakm1}
\begin{aligned}
\Delta_k(z)&=\frac{(-1)^{k-1}}{4^{k-1} (k-1)! \tilde{\phi}(z)^k} \prod_{j=1}^{k-1} (2j-1)^2\times\\
&N(z)
\begin{pmatrix}
\frac{(-1)^k}{k}\left(\frac{k}{2}-\frac{1}{4}\right) & \left(k-\frac{1}{2}\right)\ii\\[2mm]
(-1)^{k+1} \left(k-\frac{1}{2}\right)\ii  & \frac{1}{k}\left(\frac{k}{2}-\frac{1}{4}\right) & 
\end{pmatrix}
N^{-1}(z), \qquad z\in\partial D(-1,\delta),
\end{aligned}
\end{equation}
for $k\geq 1$, cf. \cite[formulas (8.5) and (8.6)]{KMcLVAV}. 

Additionally, as proved in \cite[Lemma 8.2]{KMcLVAV}, these functions $\Delta_k$ have analytic continuations to bigger discs around $z=\pm 1$ as meromorphic functions with poles at $z=\pm 1$ of order at most $[(k+1)/2]$. 

The importance of the functions $\Delta_k(z)$ is that they appear in an additive Riemann--Hilbert problem for $R_k(z)$ on the boundary of the discs. Namely, we have
\begin{equation}\label{RHforRk}
R^{(k)}_+(z)=R^{(k)}_-(z)+\sum_{j=1}^k R^{(k-j)}_-(z)\Delta_j(z), \qquad z\in\partial D(1,\delta)\cup \partial D(-1,\delta),
\end{equation}
where $+$ indicates the boundary value from the exterior and $-$ from the interior, recall Figure \ref{fig_SigmaR}. 

For $k=1$, \eqref{RHforRk} gives
\begin{equation}
R^{(1)}_+(z)=R^{(1)}_-(z)+\Delta_1(z), \qquad z\in\partial D(1,\delta)\cup \partial D(-1,\delta).
\end{equation}

Furthermore, we can write 
\begin{equation}
\begin{aligned}
 \Delta_1(z)&=\frac{A^{(1)}}{z-1}+\mathcal{O}(1), \qquad z\to 1,\\
 \Delta_1(z)&=\frac{B^{(1)}}{z+1}+\mathcal{O}(1), \qquad z\to -1,
\end{aligned}
\end{equation}
where $A^{(1)}$ are $B^{(1)}$ are constant matrices. From \eqref{Deltak1} and \eqref{Deltakm1}, and the behavior of $N(z)$, $\phi(z)$ and $\tilde{\phi}(z)$, we get
\begin{equation}
 A^{(1)}=-\frac{1}{8(2+\ii\lambda)}
	 \begin{pmatrix}
         -1 & \ii\\ \ii & 1 
         \end{pmatrix},
\qquad
 B^{(1)}=-\frac{1}{8(2-\ii\lambda)}
	 \begin{pmatrix}
         1 & \ii \\ \ii & -1 
         \end{pmatrix}.
\end{equation}

The Riemann--Hilbert problem for $R^{(1)}(z)$ can be solved as follows:
\begin{equation}
 R^{(1)}(z)=
\begin{cases}
 \displaystyle \frac{A^{(1)}}{z-1}+\frac{B^{(1)}}{z+1}, & z\in\mathbb{C}\setminus(\overline{D(1,\delta)}\cup \overline{D(-1,\delta)}),\\[2mm]
 \displaystyle \frac{A^{(1)}}{z-1}+\frac{B^{(1)}}{z+1}-\Delta_1(z), & z\in \overline{D(1,\delta)}\cup \overline{D(-1,\delta)}.
\end{cases}
\end{equation}

Hence, we have
\begin{equation}
 R^{(1)}(z)=-\frac{1}{8(2+\ii\lambda)}
	 \begin{pmatrix}
         -1 & \ii\\ \ii & 1 
         \end{pmatrix}\frac{1}{z-1}
-\frac{1}{8(2-\ii\lambda)}
	 \begin{pmatrix}
         1 & \ii \\ \ii & -1 
         \end{pmatrix}\frac{1}{z+1}, 
\end{equation}
for $z\in\mathbb{C}\setminus(\partial D(1,\delta)\cup \partial D(-1,\delta))$. Expanding for large $z$, we obtain
\begin{equation}
 R^{(1)}(z)=-\frac{\ii}{4(4+\lambda^2)}
	 \begin{pmatrix}
         \lambda & 2\\ 2 & -\lambda 
         \end{pmatrix}\frac{1}{z}
+\frac{1}{4(4+\lambda^2)}
	 \begin{pmatrix}
         2 & -\lambda \\ -\lambda & -2 
         \end{pmatrix}\frac{1}{z^2}+\mathcal{O}\left(\frac{1}{z^3}\right). 
\end{equation}

For $k=2$, equation \eqref{RHforRk} reads
\begin{equation}
 R^{(2)}_+(z)=R^{(2)}_-(z)+R^{(1)}_-(z)\Delta_1(z)+\Delta_2(z), \qquad z\in\partial D(1,\delta)\cup \partial D(-1,\delta).
\end{equation}

Now we have
\begin{equation}
\begin{aligned}
 R^{(1)}_-(z)\Delta_1(z)+\Delta_2(z)&=\frac{A^{(2)}}{z-1}+\mathcal{O}(1), \qquad z\to 1,\\
 R^{(1)}_-(z)\Delta_1(z)+\Delta_2(z)&=\frac{B^{(2)}}{z+1}+\mathcal{O}(1), \qquad z\to -1,
\end{aligned}
\end{equation}
where $A^{(2)}$ are $B^{(2)}$ are constant matrices. From \eqref{Deltak1} and \eqref{Deltakm1}, and the behavior of $N(z)$, $\phi(z)$ and $\tilde{\phi}(z)$, we get after collecting all relevant terms,
\begin{equation}
\begin{aligned}
 A^{(2)}&=-\frac{1}{64(\lambda-2\ii)^2(\lambda+2\ii)}
	 \begin{pmatrix}
         \lambda-2\ii & 4(2\ii\lambda-5)\\ -4(2\ii\lambda-5) &  \lambda-2\ii 
         \end{pmatrix},\\
 B^{(2)}&=\frac{1}{64(\lambda+2\ii)^2(\lambda-2\ii)}
	 \begin{pmatrix}
         \lambda+2\ii & -4(2\ii\lambda+5) \\ 4(2\ii\lambda+5) & \lambda+2\ii
         \end{pmatrix}.
\end{aligned}
\end{equation}

The Riemann--Hilbert problem for $R^{(2)}(z)$ can be solved as follows:
\begin{equation}
 R^{(2)}(z)=
\begin{cases}
 \displaystyle \frac{A^{(2)}}{z-1}+\frac{B^{(2)}}{z+1}, \quad z\in\mathbb{C}\setminus(\overline{D(1,\delta)}\cup \overline{D(-1,\delta)}),\\[2mm]
 \displaystyle \frac{A^{(2)}}{z-1}+\frac{B^{(2)}}{z+1}-R^{(1)}(z)\Delta_1(z)-\Delta_2(z), \quad z\in \overline{D(1,\delta)}\cup \overline{D(-1,\delta)}.
\end{cases}
\end{equation}

Reexpanding at infinity, we get
\begin{equation}
\begin{aligned}
 R^{(2)}(z)&=\frac{1}{4(4+\lambda^2)^2}
	 \begin{pmatrix}
         0 & -\ii(\lambda^2-5)\\ \ii(\lambda^2-5) & 0
         \end{pmatrix}\frac{1}{z}\\
&-\frac{1}{32(4+\lambda^2)^2}
	 \begin{pmatrix}
         \lambda^2+4 & -36\lambda \\ 36\lambda & \lambda^2+4 
         \end{pmatrix}\frac{1}{z^2}+\mathcal{O}\left(\frac{1}{z^3}\right). 
\end{aligned}
\end{equation}

If we now pick the contributions from $R^{(1)}(z)$ and $R^{(2)}(z)$ in terms of $n$, we substitute in \eqref{anbnR} and re-expand as $n\to\infty$, we obtain
\begin{equation}
 \begin{aligned}
  a_n^2&=\frac{1}{4}+\frac{4-\lambda^2}{4(4+\lambda^2)^2 n^2}+\mathcal{O}\left(\frac{1}{n^3}\right),\\
  b_n&=-\frac{2\ii\lambda}{(4+\lambda^2)^2 n^2}+\mathcal{O}\left(\frac{1}{n^3}\right).
 \end{aligned}
\end{equation}

This proves the leading terms in Theorem \ref{Th3}, and higher order coefficients can be computed if the functions $R^{(k)}(z)$ are available for $k\geq 3$. This can be accomplished in general following the idea just presented and using \eqref{RHforRk}, see also  \cite{KMcLVAV}, although the computation becomes increasingly complicated. A general algorithm to carry out this kind of computations for Jacobi--type weights will be available in \cite{DHO}, after the theory presented in \cite{KMcLVAV}. 

\appendix
\begin{section}{Steepest descent analysis for fixed $\omega$}
 For completeness, we address here briefly the case of large $n$ asymptotics with $\omega>0$ fixed, instead of coupled with $n$. In that case $\lambda=0$ and $V(z)=0$, but $w_n(z)=\ee^{-nV_n(z)}$, with potential $V_n(z)=-\frac{\ii\omega x}{n}=-\ii\lambda_n x$, that depends on $n$.

In this scenario, there are important differences in the steepest analysis. In the equilibrium problem we consider the limit potential $V(z)$ and the curve $\gamma_0=[-1,1]$, so $\phi(z)=2\log\varphi(z)$, with $\varphi(z)$ given by \eqref{varphi}. The jump of the matrix $T(z)$ on $[-1,1]$ is different from the case of coupled parameters, namely
\begin{equation}
 T_+(z)=T_-(z)
 \begin{pmatrix}
  \varphi_+(z)^{-2n} & W(z)\\ 0 & \varphi_+(z)^{2n}
 \end{pmatrix},
\end{equation}
using the variational equations \eqref{variational2}. Here we have
\begin{equation}\label{W}
  W(z)=\ee^{n(V(z)-V_n(z))}=\ee^{-n\ii z(\lambda-\lambda_n)}=\ee^{\ii\omega z}.
\end{equation}

In this situation, the factorization of the jump matrix would be
\begin{equation}
\begin{pmatrix}
  \varphi_+(z)^{-2n} & W(z)\\ 0 & \varphi_+(z)^{2n}
 \end{pmatrix}
=
\begin{pmatrix}
  1 & 0\\ \frac{\varphi_+(z)^{2n}}{W(z)} & 1
 \end{pmatrix}
\begin{pmatrix}
  0 & W(z)\\ -\frac{1}{W(z)} & 0
\end{pmatrix}
\begin{pmatrix}
  1 & 0\\ \frac{\varphi_+(z)^{-2n}}{W(z)} & 1
 \end{pmatrix}.
\end{equation}

We open a lens around $[-1,1]$ and we define:
\begin{equation}
S(z)=\begin{cases}
T(z), \qquad z\,\,\textrm{outside the lens}\\
T(z)\begin{pmatrix} 1& 0\\ -\frac{\varphi(z)^{-2n}}{W(z)} & 1\end{pmatrix}, \qquad z\,\,\textrm{in the upper part of the lens}\\
T(z)\begin{pmatrix} 1& 0\\ \frac{\varphi(z)^{-2n}}{W(z)} & 1\end{pmatrix}, \qquad z\,\,\textrm{in the lower part of the lens}.
\end{cases}
\end{equation}

Then the jumps of the matrix $S(z)$ are as follows:
\begin{equation}
 S_+(z)=S_-(z)
\begin{cases}
\begin{pmatrix}1 & 0\\ \frac{\varphi(z)^{-2n}}{W(z)} & 1\end{pmatrix}, \qquad z\in \Gamma_1\cup\Gamma_2,\\
\begin{pmatrix}0 & W(z)\\ -\frac{1}{W(z)} & 0\end{pmatrix}, \qquad z\in (-1,1). 
\end{cases}
\end{equation}

The decay of the off--diagonal element in the jump on $\Gamma_1\cup\Gamma_2$ as $n\to\infty$ is again a consequence of the properties of $\varphi(z)$, since $W(z)$ does not depend on $n$. 

When considering the model RH problem on $[-1,1]$, we seek a matrix $M(z)$ such that
\begin{enumerate}
\item $M(z)$ is analytic in $\mathbb{C}\setminus [-1,1]$.
\item On $[-1,1]$, we have the following jump:
\begin{equation}
 M_+(x)=M_-(x)
\begin{pmatrix} 0& W(z)\\ -\frac{1}{W(z)} & 0\end{pmatrix}.
\end{equation}
\item As $z\to\infty$:
\begin{equation}
 M(z)=I+\mathcal{O}\left(\frac{1}{z}\right).
\end{equation}
\end{enumerate}

It is known that the solution of this RH problem is constructed with the Szeg\H{o} function 
corresponding to the weight $W(z)$:
\begin{equation}
 D(z)=\exp\left(\frac{(z^2-1)^{1/2}}{2\pi}\int_{-1}^1\frac{\log W(s)}{\sqrt{1-s^2}}\frac{\dd s}{z-s}\right),
\end{equation}
which is analytic in $\mathbb{C}\setminus[-1,1]$ and verifies $D_+(z)D_-(z)=W(z)$ on
$(-1,1)$. Using residue calculation, in this case we can write explicitly
\begin{equation}\label{Dz_explicit}
 D(z)=\exp\left(\frac{\ii \omega}{2\varphi(z)}\right),
\end{equation}
taking the cut of the function $\varphi(z)$ on $[-1,1]$. If we further define
\begin{equation}\label{Dinf}
 D_{\infty}:=\lim_{z\to\infty} D(z)=1,
\end{equation}
 then the matrix 
\begin{equation}\label{Mz}
 M(z)=D_{\infty}^{\sigma_3}N(z)D(z)^{-\sigma_3}=N(z)D(z)^{-\sigma_3}
\end{equation}
solves the previous RH problem, and $N(z)$ is the global parametrix used in Section \ref{modelRH}. The final asymptotic results for the orthogonal polynomials can be worked out by replacing that global parametrix with $M(z)$ and making the appropriate changes in the local parametrices.

In this case the modification of the analysis is minor, since the quantity $n(\lambda-\lambda_n)=\omega$, that appears in the function $W(z)$, see \eqref{W}, and then in all subsequent computations, remains bounded independently of $n$. Similarly, one can deal with situations like $\lambda_n=\lambda+\mathcal{O}(n^{\alpha})$, with $\alpha\leq -1$. If this is not the case, for instance if $\omega=\sqrt{n}$, then the computations may become much more complicated, and the control of the error terms in the final expansion will probably involve the precise rate at which $\lambda_n$ approaches $\lambda$ as $n$ grows.
\end{section}

\section*{Acknowledgements}
The author gratefully acknowledges financial support from projects FWO G.0617.10 and FWO G.0641.11, funded by FWO (Fonds Wetenschappelijk Onderzoek, Research Fund Flanders, Belgium), and projects MTM2012--34787 and MTM2012-36732--C03--01, from the Spanish Ministry of Economy and Competitivity. The author also thanks Daan Huybrechs and Pablo Rom\'an for many stimulating and useful discussions on the topic and scope of this paper, and the two anonymous referees for comments, corrections and suggestions to improve its presentation.

\end{document}